\documentclass[11pt]{amsart}
\usepackage[letterpaper,margin=1in]{geometry}
\usepackage[T1]{fontenc}
\usepackage{lmodern}
\usepackage{amsmath,amssymb,amsthm,mathtools}

\usepackage{booktabs,array}
\usepackage{microtype}
\usepackage{needspace,etoolbox}
\usepackage[colorlinks=true,linkcolor=black,citecolor=black,urlcolor=blue]{hyperref}

\hypersetup{
  pdftitle={Complete self-shrinkers with constant norm of the second fundamental form in R4},
  pdfauthor={Yu Fu and Dan Yang},
  pdfsubject={Three-dimensional self-shrinkers and the Cheng--Wei conjecture},
  pdfkeywords={self-shrinker, second fundamental form, rigidity, Bakry-Emery Ricci tensor}
}
\allowdisplaybreaks[1]
\AtBeginDocument{%
  \setlength{\abovedisplayskip}{8pt plus 2pt minus 2pt}%
  \setlength{\belowdisplayskip}{8pt plus 2pt minus 2pt}%
  \setlength{\abovedisplayshortskip}{3pt plus 2pt minus 1pt}%
  \setlength{\belowdisplayshortskip}{6pt plus 2pt minus 1pt}%
  \clubpenalty=10000
  \widowpenalty=10000
  \displaywidowpenalty=10000
}
\BeforeBeginEnvironment{lemma}{\Needspace{4\baselineskip}}
\BeforeBeginEnvironment{proposition}{\Needspace{5\baselineskip}}
\BeforeBeginEnvironment{theorem}{\Needspace{5\baselineskip}}
\BeforeBeginEnvironment{conjecture}{\Needspace{5\baselineskip}}
\pretocmd{\subsection}{\Needspace{5\baselineskip}}{}{}
\numberwithin{equation}{section}

\newtheorem{theorem}{Theorem}[section]
\newtheorem{proposition}{Proposition}[section]
\newtheorem{lemma}{Lemma}[section]

\newtheorem{remark}{Remark}[section]
\newtheorem{conjecture}{Conjecture}[section]

\newcommand{\RR}{\mathbb R}
\newcommand{\SSS}{\mathbb S}

\newcommand{\tr}{\operatorname{tr}}

\title[Complete self-shrinkers in \(\RR^4\)]{Complete self-shrinkers with constant norm\\ of the second fundamental form in \(\RR^4\)}
\author[Y. Fu and D. Yang]{Yu Fu and Dan Yang}
\date{}

\begin{document}

    \begin{abstract}
        Let \(X:M^3\to\RR^4\) be a complete self-shrinker with constant
        squared norm \(S\) of the second fundamental form. We prove that
        \(S=0\) or \(S=1\), and that \(X(M)\) is isometric to
        a hyperplane or a standard product
        \(\SSS^k(\sqrt{k})\times\RR^{3-k}\), \(1\leq k\leq3\).
        This proves a conjecture proposed by Cheng and Wei for three-dimensional hypersurfaces.
    \end{abstract}
    \maketitle

\section{Introduction}\label{sec:introduction}

Let \(X_t:M^n\to\RR^{n+1}\) be a one-parameter family of smooth
immersions. The family satisfies the mean curvature flow if
$\partial X_t/\partial t=\vec H_t$, where $\vec H_t$ denotes the
mean curvature vector of $X_t$.
Self-shrinkers play a fundamental role in the study of singularity formation in mean curvature flow, as they arise as models of singularities and provide key insights into the geometric behavior of the flow \cite{Huisken,CM}.

An immersed hypersurface \(X:M^n\to\RR^{n+1}\) is called a
self-shrinker if
\begin{equation}\label{eq:vector-shrinker}
    \vec H+X^\perp=0,
\end{equation}
where \(X^\perp\) denotes the normal component of the position vector.
For a local unit normal \(\xi\), we use \(\bar\nabla_v\xi=-Av\), where
\(\bar\nabla\) is the Euclidean connection and \(A\) is the shape operator.
Since \(H=\tr A\) and \(\vec H=H\xi\), equation
\eqref{eq:vector-shrinker} becomes
\begin{equation*}
    H+\langle X,\xi\rangle=0.
\end{equation*}
We denote by \(S=|A|^2\) the squared norm of the second fundamental
form and by \(R=H^2-S\) the scalar curvature.

The hyperplane through the origin, the round sphere
\(\SSS^n(\sqrt n)\), and the cylinders \(\SSS^k(\sqrt k)\times\RR^{n-k}\), \(1\leq k\leq n-1\),
are self-shrinkers with constant \(S\). We have \(S=0\) for the
hyperplane and \(S=1\) for the sphere and cylinders.
Abresch--Langer \cite{AL} classified closed self-shrinking curves
in \(\RR^2\).
Brendle \cite{Brendle} proved that a compact embedded self-shrinker
in \(\RR^3\) of genus zero is a round sphere.
Drugan \cite{Drugan} constructed an immersed, nonembedded
\(\SSS^2\) self-shrinker in \(\RR^3\).

Some important gap and rigidity results for self-shrinkers were obtained by Cao--Li
\cite{CaoLi}, Cheng--Peng \cite{CP}, Ding--Xin \cite{DX}, and
Cheng--Wei \cite{CWgap}.
Li--Wei \cite{LWrigidity} proved classification and rigidity theorems
in higher codimension.
Ding \cite{DingRigidity} obtained rigidity under an integral
curvature condition. These results establish conditions under which self-shrinkers must have the standard geometric models, providing important criteria for their classification.
Cheng--Wei proposed the following conjecture, which is stated in
\cite{CWY}.

\Needspace{8\baselineskip}
\begin{conjecture}\label{conj:Cheng-Wei}
    Let \(X:M^n\to\RR^{n+1}\) be a complete self-shrinker with
    constant squared norm of the second fundamental form. Then the
    self-shrinker $X(M^n)$ is isometric to a hyperplane \(\RR^n\) through the
    origin, the round sphere \(\SSS^n(\sqrt n)\) centered at the
    origin, or a cylinder
\[
  \SSS^k(\sqrt k)\times\RR^{n-k},\qquad 1\leq k\leq n-1.
\]
In particular, \(S=0\) or \(S=1\).
\end{conjecture}

For two-dimensional complete self-shrinkers with constant \(S\),
Ding--Xin \cite{DX} obtained the classification under polynomial
volume growth. Cheng--Ogata \cite{CO} removed this assumption.
Thus the conjecture holds for \(n=2\).
For \(n=3\), Cheng--Li--Wei \cite{CLWf4} proved the
conjecture under the additional assumption that \(f_4=\tr A^4\) is constant. The same authors \cite{CLWf3} obtained the classification under the assumption that \(f_3=\tr A^3\) is constant.

In arbitrary dimension, Cheng--Wei--Yano \cite{CWY} obtained a
second-gap theorem for complete self-shrinkers with constant \(S\) and \(f_3\).
Cheng--Li--Wei \cite{CLWscalar} proved that \(S=0\) or \(S=1\)
when the constant \(S\) satisfies \(S\leq7/5\).
For other rigidity and gap results, see \cite{DXvolume, CL,FY,Rimoldi, LWvolume, Halldorsson}.

Guo \cite{Guo} proved a scalar curvature gap theorem for compact
self-shrinkers.
Luo--Sun--Yin \cite{LSY} obtained classification results for complete
self-shrinkers with nonnegative constant scalar curvature, assuming
polynomial volume growth in the positive case.
Cheng--Li--Wei \cite{CLW} also classified complete self-shrinkers
with positive constant scalar curvature. Since
\(R=H^2-S\), the constancy of \(R\) differs from the constancy
of \(S\). In fact, when \(S>1\) is constant, the
\(\mathcal L\)-maximum principle gives \(\inf_M R=-S\), as
shown in Lemma~\ref{lem:weighted-geometry}. Hence a positive lower
bound for \(R\) cannot be assumed in this case.

In this paper, we study complete three-dimensional self-shrinkers
with constant norm of the second fundamental form. We remove the assumption
that either \(f_3\) or \(f_4\) is constant and obtain the
following theorem.

\begin{theorem}\label{thm:main}
    Let \(X:M^3\to\RR^4\) be a connected complete immersed
    self-shrinking hypersurface without boundary, with constant
    squared norm \(S\) of the second fundamental form. Then
    either \(S\equiv0\) and \(X(M^3)\) is isometric to
    a hyperplane through the origin, or \(S\equiv1\) and
    \(X(M)\) is isometric to
    \(\SSS^k(\sqrt{k})\times\RR^{3-k}\) for some
    \(k\in\{1,2,3\}\).
\end{theorem}

\begin{remark}
	Theorem~\ref{thm:main} gives the classification in
Conjecture~\ref{conj:Cheng-Wei} for \(n=3\).
We do not require properness or polynomial volume growth in
Theorem~\ref{thm:main}. For complete immersed self-shrinkers,
properness, polynomial volume growth, and finiteness of Gaussian
volume are equivalent \cite{CZ}.
\end{remark}
To prove Theorem~\ref{thm:main}, we first observe that the Simons
identity gives $\sum_{i,j,k}h_{ijk}^2=S(S-1)$. Since the left-hand
side is nonnegative, it follows that $S=0$ or $S\geq1$. We therefore
need only exclude $S>1$. Using the derivative identities in
Section~\ref{sec:identities}, we first obtain uniform bounds for
\(h_{ijkl}\). We then prove that \(|X|\) is bounded. Applying
the Hessian maximum principle to \(f=|X|^2/2\), we obtain limiting
components satisfying \(H\lambda_i\geq1\) for every \(i\).
After changing the normal if necessary, the limiting principal
curvatures are all positive.

We solve the trace equations and express the remaining components of
$h_{ijkl}$ in terms of four free variables. Substitution into a linear
combination of the two Bochner identities yields a quadratic expression.
We use the trace equations to eliminate the constrained components and minimize the resulting quadratic expression. The remaining positivity estimates are established separately for distinct and repeated principal curvatures.
For distinct principal curvatures, a concavity argument
gives the required boundary estimate. Eight remaining coefficient
inequalities are verified by exact rational arithmetic. The
polynomial formulas are given in Appendix~\ref{app:polynomials}.

The paper is organized as follows.
In Section~\ref{sec:identities}, we establish the identities
for the second fundamental form and its covariant derivatives.
In Section~\ref{sec:global}, we obtain the derivative bounds
and the limiting relations used in the proof.
Section~\ref{sec:algebra} is devoted to the algebraic inequality.
In Section~\ref{sec:classification}, we complete the proof
of Theorem~\ref{thm:main}.
The polynomial coefficient formulas and a rational calculation
are given in Appendix~\ref{app:polynomials}.

\section{Preliminaries}\label{sec:identities}
Let $X:M^3\to\mathbb R^4$ be a connected complete immersed self-shrinker
with constant norm of the second fundamental form. We denote by
$\bar\nabla$ the Euclidean connection on $\mathbb R^4$ and by $\nabla$ the
Levi-Civita connection on $M$. If $M$ is not orientable, we pass to its
oriented double cover and lift the immersion. The lifted metric is
complete, the self-shrinker equation and the value of $S$ are unchanged,
and the immersed image is the same. We may therefore assume that $M$
is orientable and fix a global unit normal vector field $\xi$.
For tangent vector fields $Y,Z\in TM$, the shape operator $A$ is defined by
$
\bar\nabla_Y\xi=-AY,
$
and the second fundamental form is given by
$
h(Y,Z)=\langle AY,Z\rangle.$
Let $\{e_1,e_2,e_3\}$ be a local orthonormal frame on $M$. We write
\[
h_{ij}=h(e_i,e_j),\qquad
H=\sum_i h_{ii},\qquad
S=|A|^2=\sum_{i,j}h_{ij}^2.
\]
We denote the components of $\nabla h$ and $\nabla^2h$ by $h_{ijk}$ and
$h_{ijkl}$, respectively. The Gauss and Codazzi equations are
\[
  R_{ijkl}=h_{il}h_{jk}-h_{ik}h_{jl}, \qquad h_{ijk}=h_{ikj}.
\]
In particular, $h_{ijk}$ is symmetric in all three indices, and
$h_{ijkl}$ is symmetric in the first three indices. We use
$h_{ijkl}=(\nabla_l\nabla h)_{ijk}$ and the same order for higher
covariant derivatives. With the Gauss convention above, the Ricci identities are
\begin{equation}\label{eq:ricci-general}
\begin{aligned}
h_{ijkl}-h_{ijlk}
 &=-\sum_m\bigl(h_{mj}R_{mikl}+h_{im}R_{mjkl}\bigr),\\
h_{ijklm}-h_{ijkml}
 &=-\sum_r\bigl(h_{rjk}R_{rilm}+h_{irk}R_{rjlm}+h_{ijr}R_{rklm}\bigr).
\end{aligned}
\end{equation}

We use the drift Laplacian
$
\mathcal L=\Delta-\langle X,\nabla\cdot\rangle,
$
where $\Delta$ denotes the Laplacian.
Differentiating the self-shrinker equation yields
\begin{equation}\label{eq:Hderivatives}
\begin{aligned}
  H_i&=\sum_j h_{ij}\langle X,e_j\rangle,\\
  H_{ij}&=\sum_k h_{ijk}\langle X,e_k\rangle
           +h_{ij}-H\sum_k h_{ik}h_{kj}.
\end{aligned}
\end{equation}
The Simons identities for a self-shrinker are given by
\begin{equation}
\label{eq:Simons}
\mathcal L h_{ij}=(1-S)h_{ij},\qquad
\frac12\mathcal L S
=\sum_{i,j,k}h_{ijk}^2+S(1-S).
\end{equation}
These formulas are given in \cite{CP}. Since $S$ is constant, it follows that
\begin{equation}
\label{eq:nablah}
\sum_{i,j,k}h_{ijk}^2=S(S-1).
\end{equation}
In particular, either $S=0$ or $S\geq1$.

Let \(f=\frac{|X|^2}{2}\).
Then
$
f_i=\langle X,e_i\rangle.
$
Using the Gauss formula and the self-shrinker equation, we obtain
\[
  f_{ij} = \delta_{ij}+\langle X,\xi\rangle h_{ij} = \delta_{ij}-Hh_{ij}.
\]
By the Gauss equation,
\[
  R_{ij} = Hh_{ij}-\sum_k h_{ik}h_{kj}, \qquad R=H^2-S.
\]
Hence
\begin{equation}
\label{eq:Ricf}
(\operatorname{Ric}_f)_{ij}
:=R_{ij}+f_{ij}
=
\delta_{ij}-\sum_k h_{ik}h_{kj}.
\end{equation}
At a fixed point, choose an orthonormal frame diagonalizing $A$, so that
$h_{ij}=\lambda_i\delta_{ij}$, and
\[
  (\operatorname{Ric}_f)_{ij} = (1-\lambda_i^2)\delta_{ij}.
\]

\begin{lemma}\label{lem:weighted-geometry}
Let $X:M^3\to\mathbb R^4$ be a complete self-shrinker with constant
$S>1$. Then
\[
  \inf_M H^2=0,\qquad \inf_M R=-S.
\]
\end{lemma}

\begin{proof}
Suppose that $\inf_M H^2>0$. After reversing the unit normal if necessary,
we may assume that $H>0$. Since $H^2\leq 3S$, we have \(0<\inf_M H\leq H\leq \sqrt{3S}\).
In a principal frame, the Gauss equation gives
$
R_{ii}=H\lambda_i-\lambda_i^2.
$
Since $|H|\leq\sqrt{3S}$ and $|\lambda_i|\leq\sqrt S$, the Ricci curvature of
$M$ is bounded from below. Applying the $\mathcal L$-maximum principle of
Cheng--Peng \cite{CP} to $-H$, we
obtain a sequence $\{p_m\}\subset M$ such that
\[
  H(p_m)\longrightarrow\inf_M H,\qquad \limsup_{m\to\infty}\mathcal L(-H)(p_m)\leq0.
\]

On the other hand, taking the trace of \eqref{eq:Simons} gives
\[
  \mathcal L(-H)=(S-1)H \geq(S-1)\inf_M H>0,
\]
a contradiction. Hence $\inf_M H^2=0$. Since $R=H^2-S$, we also obtain \(\inf_M R=-S\).
\end{proof}

\begin{remark}
By \eqref{eq:Ricf} and $\lambda_i^2\leq S$, we have \(\operatorname{Ric}_f\geq(1-S)g\).
When $S>1$, this does not provide a positive lower bound for $\operatorname{Ric}_f$.
Moreover, Lemma~\ref{lem:weighted-geometry} gives $\inf_M R=-S$, so the scalar curvature cannot have a
positive lower bound in the present case. We shall use the Hessian
maximum principle instead to obtain the inequalities $H\lambda_i\geq1$.
For comparison results for the Bakry--\'Emery Ricci tensor, see \cite{WeiWylie}.
\end{remark}

\begin{lemma}
\label{lem:bochner1}
Let $X:M^3\to\mathbb R^4$ be a complete self-shrinker with constant
$S$. Then
\begin{equation}
\label{eq:bochner1}
\sum_{i,j,k,l}h_{ijkl}^2
=
S(S-1)(S-2)
+
3\sum_{i,j,k}
(\lambda_i^2-2\lambda_i\lambda_j)h_{ijk}^2.
\end{equation}
\end{lemma}

\begin{proof}
By the standard Bochner identity for self-shrinkers,
\[
|\nabla^2 h|^2
=
\frac12\mathcal L|\nabla h|^2
+(S-2)|\nabla h|^2
+3\sum_{i,j,k}
(\lambda_i^2-2\lambda_i\lambda_j)h_{ijk}^2
+\frac32|\nabla S|^2.
\]
See also \cite{LeiXuXu}.
Since $S$ is constant, $\nabla S=0$, and by \eqref{eq:nablah}, \(|\nabla h|^2=S(S-1)\),
which is also constant. Hence \(\mathcal L|\nabla h|^2=0\).
Substituting these identities into the preceding formula gives
\[
  |\nabla^2 h|^2 = S(S-1)(S-2) + 3\sum_{i,j,k} (\lambda_i^2-2\lambda_i\lambda_j)h_{ijk}^2,
\]
which yields \eqref{eq:bochner1}.
\end{proof}

At a fixed point, we work in a principal frame
$h_{ij}=\lambda_i\delta_{ij}$ and set
\[
f_3=\sum_i\lambda_i^3.
\]

\begin{lemma}
Let $X:M^3\to\mathbb R^4$ be a complete self-shrinker such that
\[
S=\sum_{i,j}h_{ij}^2
\qquad\text{and}\qquad
\sum_{i,j,k}h_{ijk}^2
\]
are constant. Then
\begin{equation}\label{17}
\sum_{i,j,k,l}\lambda_l h_{ijkl}^2
-3\sum_{i,j,k,a,b}h_{ijk}h_{abk}h_{ijab}
+\sum_{i,j,k}
\left(
3S\lambda_i+6\lambda_i^2\lambda_j-3\lambda_i^3-f_3
\right)h_{ijk}^2=0.
\end{equation}
\end{lemma}

\begin{proof}
Since $S$ is constant, it follows that
$
  0=\frac12 S_i=\sum_{a,b}h_{ab}h_{abi}.
$
Differentiating with respect to $e_j$ gives
\[
  \sum_{a,b}h_{ab}h_{abij} = -\sum_{a,b}h_{abi}h_{abj}.
\]
Using the Codazzi equation, the Ricci identity, and the Gauss equation, we obtain
\begin{equation}
\label{eq:hijab}
\sum_{a,b}h_{ab}h_{ijab}
=
-\sum_{a,b}h_{abi}h_{abj}
+
S\sum_l h_{il}h_{lj}
-f_3h_{ij}.
\end{equation}

We now differentiate \eqref{eq:hijab} in the direction $e_k$, multiply by
$h_{ijk}$, and sum over $i,j,k$. We compute the two sides separately.

For the left-hand side, the product rule gives
\begin{equation}\label{eq:left-derivative}
\begin{aligned}
 \sum_{i,j,k}h_{ijk}\nabla_k\left(\sum_{a,b}h_{ab}h_{ijab}\right)
  &=\sum_{i,j,k,a,b}h_{ijk}h_{abk}h_{ijab}
    +\sum_{i,j,k,a,b}h_{ijk}h_{ab}h_{ijabk}\\
  &=\sum_{i,j,k,a,b}h_{ijk}h_{abk}h_{ijab}
    +\sum_{i,j,k,a}\lambda_a h_{ijk}h_{ijaak}.
\end{aligned}
\end{equation}

For the right-hand side of \eqref{eq:hijab}, differentiating and contracting
with $h_{ijk}$ gives
\[
\begin{aligned}
 &-\sum_{i,j,k,a,b}h_{ijk}
     \left(h_{abik}h_{abj}+h_{abi}h_{abjk}\right)
  +S\sum_{i,j,k,l}h_{ijk}
     \left(h_{ilk}h_{lj}+h_{il}h_{ljk}\right)\\
 &-\sum_{i,j,k}(f_3)_k h_{ij}h_{ijk}
  -f_3\sum_{i,j,k}h_{ijk}^2.
\end{aligned}
\]
By the full symmetry of $h_{ijk}$ and the symmetry of $h_{ijkl}$ in the
first three indices, the first term becomes
$
  -2\sum_{i,j,k,a,b}h_{ijk}h_{abk}h_{ijab}.
$
At the chosen point,
\[
  \sum_l h_{ilk}h_{lj}=\lambda_jh_{ijk}, \qquad \sum_l h_{il}h_{ljk}=\lambda_i h_{ijk},
\]
and therefore, using again the symmetry of $h_{ijk}$,
\[
S\sum_{i,j,k,l}
h_{ijk}
\left(
h_{ilk}h_{lj}
+
h_{il}h_{ljk}
\right)
=
2S\sum_{i,j,k}\lambda_i h_{ijk}^2.
\]
Moreover,
\[
  \sum_{i,j}h_{ij}h_{ijk} = \frac12 S_k = 0,
\]
so the term involving $(f_3)_k$ vanishes. Hence
\begin{equation}\label{eq:right-derivative}
\begin{aligned}
 \sum_{i,j,k}h_{ijk}\nabla_k\left(\sum_{a,b}h_{ab}h_{ijab}\right)
  =-2\sum_{i,j,k,a,b}h_{ijk}h_{abk}h_{ijab}+\sum_{i,j,k}(2S\lambda_i-f_3)h_{ijk}^2.
\end{aligned}
\end{equation}

Comparing \eqref{eq:left-derivative} and \eqref{eq:right-derivative}, we obtain
\begin{equation}
\label{eq:ijaak}
\begin{aligned}
\sum_{i,j,k,a}\lambda_a h_{ijk}h_{ijaak}
={}&
-3\sum_{i,j,k,a,b}
h_{ijk}h_{abk}h_{ijab}+
\sum_{i,j,k}
(2S\lambda_i-f_3)h_{ijk}^2.
\end{aligned}
\end{equation}

We next replace $h_{ijaak}$ by $h_{ijkaa}$. For this purpose, we compute
the commutator \(h_{ijkaa}-h_{ijaak}\).
Using the Ricci identity \eqref{eq:ricci-general}, we obtain
\begin{equation}
\label{eq:commutator}
\begin{aligned}
\sum_a\lambda_a(h_{ijkaa}-h_{ijaak})
=
\sum_{a,b}\lambda_a\Big[
&(\nabla_aR_{kaib})h_{bj}
+(\nabla_aR_{kajb})h_{ib}\\
&+2R_{kaib}h_{bja}
+2R_{kajb}h_{iba}+R_{kaab}h_{ijb}
\Big].
\end{aligned}
\end{equation}
Since $S$ is constant, \(\sum_a\lambda_a h_{aaj}=0\).
Substituting these identities into \eqref{eq:commutator}, multiplying by
$h_{ijk}$, summing over $i,j,k$, using the Gauss equation and the full symmetry of $h_{ijk}$,
we obtain
\begin{equation}
\label{eq:commutator-contraction}
\begin{aligned}
\sum_{i,j,k,a}
\lambda_a h_{ijk}
(h_{ijkaa}-h_{ijaak})
=
\sum_{i,j,k}
\left(
S\lambda_i
+6\lambda_i^2\lambda_j
-3\lambda_i^3
\right)h_{ijk}^2.
\end{aligned}
\end{equation}
Adding \eqref{eq:commutator-contraction} to \eqref{eq:ijaak}, we obtain
\begin{equation}\label{eq:ijk-aa}
\begin{aligned}
 \sum_{i,j,k,a}\lambda_a h_{ijk}h_{ijkaa}
  &=-3\sum_{i,j,k,a,b}h_{ijk}h_{abk}h_{ijab}\\
  &\quad+\sum_{i,j,k}
      \left(3S\lambda_i+6\lambda_i^2\lambda_j-3\lambda_i^3-f_3\right)h_{ijk}^2.
\end{aligned}
\end{equation}

Finally, since $\sum_{i,j,k}h_{ijk}^2$ is constant, we have
\[
  0 = \frac12 \sum_l\lambda_l \nabla_l\nabla_l \left( \sum_{i,j,k}h_{ijk}^2 \right).
\]
Expanding the derivatives gives
\begin{equation}
\label{eq:weighted-nablah}
\sum_{i,j,k,l}\lambda_l h_{ijk}h_{ijkll}=-\sum_{i,j,k,l}\lambda_l h_{ijkl}^2.
\end{equation}
Combining \eqref{eq:ijk-aa} and  \eqref{eq:weighted-nablah} yields
\eqref{17}.
\end{proof}
\section{A limiting argument}\label{sec:global}

We now turn to the limiting argument that produces the algebraic constraints used in the next section. Throughout this section we assume \(S>1\).

\begin{lemma}\label{lem:bounds}
Let $X:M^3\to\mathbb R^4$ be a complete self-shrinker
with constant $S>1$. Then the absolute values of
\[
h_{ij},\quad
h_{ijk},\quad
h_{ijkl},\quad
H_{ij},\quad
\sum_k h_{ijk}\langle X,e_k\rangle
\]
are bounded by a constant depending only on $S$.
\end{lemma}

\begin{proof}
We have \(|\lambda_i|\leq\sqrt S\) and
\(\sum h_{ijk}^2=S(S-1)\). Therefore
\eqref{eq:bochner1} gives
\[
  \sum_{i,j,k,l}h_{ijkl}^2 \leq\{S-2+9S\}S(S-1)=(10S-2)S(S-1).
\]
The identity \(H_{ij}=\sum_kh_{kkij}\) bounds the second
derivatives of \(H\). By \eqref{eq:Hderivatives},
\[
  \sum_kh_{ijk}\langle X,e_k\rangle =H_{ij}-h_{ij}+H\sum_kh_{ik}h_{kj}.
\]
Since \(|H|\leq\sqrt{3S}\), the last expression is bounded
in terms of \(S\).
\end{proof}

\begin{proposition}\label{prop:bounded}
Let $X:M^3\to\mathbb R^4$ be a complete self-shrinker
with constant $S>1$. Then $|X|$ is bounded on $M$.
\end{proposition}

\begin{proof}
We argue by contradiction. Suppose that $|X|$ is unbounded on $M$.
For each positive integer $m$, choose $p_m\in M$ such that \(|X(p_m)|\ge m^2\).
Consider the function \(F_m(q)=H(q)-m^{-1}d(p_m,q)\),
where $d$ denotes the intrinsic distance on $M$.

Since $|H|\le \sqrt{3S}$, the function $H$ is bounded. Hence
\[
  F_m(q)\to -\infty \qquad\text{as }\qquad d(p_m,q)\to\infty.
\]
Therefore the maximum of $F_m$ is attained in a sufficiently large
closed metric ball centered at $p_m$. Since $M$ is complete, closed
metric balls are compact by the Hopf--Rinow theorem. Thus there exists
$q_m\in M$ such that
\(F_m(q_m)=\max_M F_m\).

Since $F_m(q_m)\geq F_m(p_m)$, we obtain
\[
  H(q_m)-m^{-1}d(p_m,q_m)\ge H(p_m).
\]
Since $|H|\le\sqrt{3S}$, it follows that
\[
  d(p_m,q_m) \le m\bigl(H(q_m)-H(p_m)\bigr) \le 2\sqrt{3S}\,m.
\]

We next estimate the gradient of $H$ at $q_m$. Let
$v\in T_{q_m}M$ be a unit vector and let $\gamma$ be the unit-speed
geodesic satisfying $\gamma(0)=q_m$ and $\gamma'(0)=v$.
By the maximality of $q_m$, \(F_m(\gamma(t))\le F_m(q_m)\),
and hence
\[
  H(\gamma(t))-H(q_m) \le m^{-1} \left\{ d(p_m,\gamma(t))-d(p_m,q_m) \right\}.
\]
By the triangle inequality,
\[
  d(p_m,\gamma(t)) \le d(p_m,q_m)+d(q_m,\gamma(t)) \le d(p_m,q_m)+|t|,
\]
and therefore \(H(\gamma(t))-H(q_m)\le m^{-1}|t|\).
Dividing by $t$ and letting $t\to0^+$ and $t\to0^-$,
respectively, we obtain
\[
dH_{q_m}(v)\le m^{-1},
\qquad
dH_{q_m}(v)\ge -m^{-1}.
\]
Since $v$ is arbitrary, it follows that
\begin{equation}\label{eq:gradHsmall}
|\nabla H(q_m)|\le m^{-1}.
\end{equation}
The Euclidean distance $|X(p_m)-X(q_m)|$ is no larger than
$d(p_m,q_m)$. Hence
\[
\begin{aligned}
|X(q_m)|\ge
|X(p_m)|-|X(q_m)-X(p_m)|\ge
|X(p_m)|-d(p_m,q_m)\ge
m^2-2\sqrt{3S}\,m,
\end{aligned}
\]
and consequently \(|X(q_m)|\to\infty\).
On the other hand, by the self-shrinker equation, \(X^\perp=-H\xi\),
so \(|X^\perp|=|H|\le\sqrt{3S}\).
Therefore
\[
  |X^\top(q_m)|\to\infty.
\]

After discarding finitely many points, choose an orthonormal frame at each $q_m$ such that
\[
  e_3=\frac{X^\top(q_m)}{|X^\top(q_m)|}.
\]
By Lemma~\ref{lem:bounds}, the components $h_{ij}$, $h_{ijk}$, and $h_{ijkl}$ are
uniformly bounded. Passing to a subsequence, we may therefore assume
that
\[
h_{ij}(q_m)\to\bar h_{ij},
\qquad
h_{ijk}(q_m)\to\bar h_{ijk},
\qquad
h_{ijkl}(q_m)\to\bar h_{ijkl}.
\]

At $q_m$, the first identity in \eqref{eq:Hderivatives} becomes
\[
  H_i(q_m) = h_{i3}(q_m)|X^\top(q_m)|.
\]
By \eqref{eq:gradHsmall}, \(|H_i(q_m)|\le |\nabla H(q_m)|\to0\),
whereas $|X^\top(q_m)|\to\infty$. Hence \(h_{i3}(q_m)\to0\),
and therefore
$\bar h_{i3}=0$ for all $i$.

Similarly, Lemma~\ref{lem:bounds} shows that \(\sum_k h_{ijk}\langle X,e_k\rangle\)
is uniformly bounded. Since our choice of frame gives
\[
  \sum_k h_{ijk}\langle X,e_k\rangle = |X^\top(q_m)|\,h_{ij3}(q_m),
\]
we obtain \(h_{ij3}(q_m)\to 0\).
Thus \(\bar h_{ij3}=0\).
Since $h_{ijk}$ is fully symmetric, every component
$\bar h_{ijk}$ containing the index $3$ vanishes.
After rotating the $(e_1,e_2)$-plane, we may assume that
\[
  (\bar h_{ij}) = \operatorname{diag} (\bar\lambda_1,\bar\lambda_2,0).
\]
We now use the two trace relations satisfied by $h_{ijk}$. Since \(H_k=\sum_i h_{iik}\) and \eqref{eq:gradHsmall} gives $H_k(q_m)\to0$, we obtain, for
$k=1,2$, \(\bar h_{11k}+\bar h_{22k}=0\),
because $\bar h_{33k}=0$.

Moreover, since $S$ is constant, \(0=\frac12S_k = \sum_{i,j}h_{ij}h_{ijk}\).
Passing to the limit and using the diagonal form of $\bar h_{ij}$ gives
\[
  \bar\lambda_1\bar h_{11k} + \bar\lambda_2\bar h_{22k} =0, \quad k=1,2.
\]

We now distinguish two cases.

If \(\bar\lambda_1\ne\bar\lambda_2\),
then, for each $k=1,2$, the system
\[
\bar h_{11k}+\bar h_{22k}=0,
\quad
\bar\lambda_1\bar h_{11k}
+
\bar\lambda_2\bar h_{22k}=0
\]
has only the trivial solution. Hence $\bar h_{11k} = \bar h_{22k} =0$, $k=1,2$.
Together with the vanishing of all components containing the index $3$
and the full symmetry of $\bar h_{ijk}$, this implies $\bar h_{ijk}=0$ for all $i,j,k$.
This contradicts \eqref{eq:nablah}, since
\[
  \sum_{i,j,k}\bar h_{ijk}^2 = S(S-1)>0.
\]

If \(\bar\lambda_1=\bar\lambda_2\), then since $\bar\lambda_3=0$ and \(\bar\lambda_1^2+\bar\lambda_2^2=S\),
we have \(\bar\lambda_1^2 = \bar\lambda_2^2 = \frac{S}{2}\).
As noted above, every nonzero component $\bar h_{ijk}$ has
$i,j,k\in\{1,2\}$. Hence, whenever $\bar h_{ijk}\ne0$,
\[
  \bar\lambda_i^2 - 2\bar\lambda_i\bar\lambda_j = -\frac{S}{2}.
\]
Passing to the limit in \eqref{eq:bochner1}, we obtain
\[
\begin{aligned}
 \sum_{i,j,k,l}\bar h_{ijkl}^2
 &=S(S-1)(S-2)-\frac{3S}{2}\sum_{i,j,k}\bar h_{ijk}^2\\
 &=\left(S-2-\frac{3S}{2}\right)S(S-1)
   =-\left(\frac S2+2\right)S(S-1)<0.
\end{aligned}
\]
This is impossible, since the left-hand side is a sum of squares. Thus $|X|$ is bounded on $M$.
\end{proof}
We shall use the following form of Omori's Hessian maximum
principle \cite[Theorem A']{Omori}. Let $M$ be a complete
Riemannian manifold with sectional curvature bounded below.
If $f\in C^2(M)$ is bounded above, then there exists a sequence
$\{q_m\}\subset M$ such that
\[
\begin{gathered}
f(q_m)>\sup_M f-\frac{1}{m},
|\nabla f(q_m)|<\frac{1}{m},
\qquad
\nabla^2 f(q_m)\leq \frac{1}{m}g.
\end{gathered}
\]

\Needspace{16\baselineskip}
\begin{proposition}
\label{prop:limit}
Let $X:M^3\to\mathbb R^4$ be a complete self-shrinker
with constant $S>1$. Then there exists a sequence of points
$\{q_m\}\subset M$ and a choice of orthonormal principal frames
at $q_m$ such that, after passing to a subsequence,
the principal curvatures $\lambda_i$ and the components
$h_{ijk},h_{ijkl}$ converge to
$\bar\lambda_i,\bar h_{ijk},\bar h_{ijkl}$, respectively,
satisfying
\begin{equation}
\label{eq:limit-principal}
\begin{gathered}
 \bar\lambda_i>0,\qquad \bar H=\sum_i\bar\lambda_i,\qquad
 \bar H\bar\lambda_i\geq1,\\
 \sum_i\bar\lambda_i^2=S,\qquad
 \sum_{i,j,k}\bar h_{ijk}^2=S(S-1).
\end{gathered}
\end{equation}
Moreover,
\begin{equation}
\label{eq:limit-firsttrace}
\sum_i\bar h_{iik}
=
\sum_i\bar\lambda_i\bar h_{iik}
=
0,
\end{equation}
and
\begin{equation}
\label{eq:limit-secondtrace}
\sum_i\bar h_{iikl}
=
(\bar\lambda_k-\bar H\bar\lambda_k^2)\delta_{kl},
\end{equation}
while
\begin{equation}
\label{eq:limit-weightedtrace}
\sum_i\bar\lambda_i\bar h_{iikl}
=
-\bar Q_{kl},
\qquad
\bar Q_{kl}
=
\sum_{i,j}\bar h_{ijk}\bar h_{ijl}.
\end{equation}
The components $\bar h_{ijk}$ are fully symmetric, and
$\bar h_{ijkl}$ are symmetric in the first three indices.
The Ricci identity in the first line of \eqref{eq:ricci-general}
and the identities \eqref{eq:bochner1} and \eqref{17} also hold
for these limiting components.
\end{proposition}

\begin{proof}
By Proposition~\ref{prop:bounded}, the function \(f=\frac{|X|^2}{2}\)
is bounded above on $M$. We first verify that the Hessian maximum
principle applies to $f$.

Let $e_i,e_j$ be orthonormal tangent vectors. By the Gauss equation,
\[
  K(e_i,e_j)=h_{ii}h_{jj}-h_{ij}^2.
\]
Using \(h_{ii}h_{jj} \ge -\frac{h_{ii}^2+h_{jj}^2}{2}\),
we obtain
\[
  K(e_i,e_j) \ge -\frac{h_{ii}^2+h_{jj}^2}{2}-h_{ij}^2.
\]
Since \(h_{ii}^2+h_{jj}^2+2h_{ij}^2\le S\),
it follows that \(K(e_i,e_j)\ge-\frac{S}{2}\).
Thus the sectional curvature of $M$ is bounded from below, and hence
the Hessian maximum principle applies to $f$.
Therefore there exist points $q_m\in M$ and positive numbers
$\varepsilon_m\to0$ such that
\[
  f(q_m)\to\sup_M f, \qquad |\nabla f(q_m)|\to0, \qquad \nabla^2f(q_m)\le\varepsilon_m g.
\]

Recall that
$f_i=\langle X,e_i\rangle$, $f_{ij}=\delta_{ij}-Hh_{ij}$.
Hence
\[
  |\nabla f|^2 = \sum_i\langle X,e_i\rangle^2 = |X^\top|^2
\]
and therefore
\(|X^\top(q_m)|\to0\).

At each $q_m$, choose an orthonormal frame diagonalizing the second
fundamental form, so that \(h_{ij}(q_m)=\lambda_i(q_m)\delta_{ij}\).
The inequality \(\nabla^2f(q_m)\le\varepsilon_m g\)
then gives, for each $i$, \(1-H(q_m)\lambda_i(q_m)\le\varepsilon_m\),
or equivalently,
\begin{equation}
\label{eq:Hlambda-approx}
H(q_m)\lambda_i(q_m)\ge1-\varepsilon_m.
\end{equation}

Taking the trace of \((\delta_{ij}-Hh_{ij})(q_m)\le\varepsilon_m(\delta_{ij})\),
we obtain \(3-H(q_m)^2\le3\varepsilon_m\),
and hence \(H(q_m)^2\ge3(1-\varepsilon_m)\).
After passing to a subsequence and reversing the unit normal if necessary,
we may assume that \(H(q_m)>0\)
for all sufficiently large $m$. Since $\varepsilon_m\to0$, it follows
from \eqref{eq:Hlambda-approx} that \(\lambda_i(q_m)>0\)
for all sufficiently large $m$.

By Lemma~\ref{lem:bounds}, the components $h_{ij}$, $h_{ijk}$, and $h_{ijkl}$ are
uniformly bounded in terms of $S$. Passing to a further subsequence, we
may therefore assume that
\[
\lambda_i(q_m)\to\bar\lambda_i,
\qquad
h_{ijk}(q_m)\to\bar h_{ijk},
\qquad
h_{ijkl}(q_m)\to\bar h_{ijkl}.
\]
Set \(\bar H = \bar\lambda_1+\bar\lambda_2+\bar\lambda_3\).
Passing to the limit in \eqref{eq:Hlambda-approx} gives \(\bar H\bar\lambda_i\ge1\).
Moreover, \(H^2\le3\sum_i\lambda_i^2=3S\),
so \(\bar H\le\sqrt{3S}\).
Since $\bar H\bar\lambda_i\ge1$, we have $\bar H>0$ and
\[
  \bar\lambda_i \ge \frac1{\bar H} \ge \frac1{\sqrt{3S}} >0.
\]
Finally, passing to the limit in \(\sum_i\lambda_i^2=S\)
and in \eqref{eq:nablah} gives
\[
  \sum_i\bar\lambda_i^2=S, \qquad \sum_{i,j,k}\bar h_{ijk}^2=S(S-1).
\]
Thus \eqref{eq:limit-principal} holds.

We next derive the first-order trace relations. By the first identity in
\eqref{eq:Hderivatives},
\[
  H_k = \sum_jh_{kj}\langle X,e_j\rangle.
\]
Since $h_{kj}$ is uniformly bounded and \(|X^\top(q_m)|\to0\),
we obtain \(H_k(q_m)\to0\).
On the other hand, \(H_k=\sum_i h_{iik}\),
and hence
\(\sum_i\bar h_{iik}=0\).

Since $S$ is constant,
\[
  0 = \frac12S_k = \sum_{i,j}h_{ij}h_{ijk}.
\]
At $q_m$, where $h_{ij}=\lambda_i\delta_{ij}$, this becomes \(\sum_i\lambda_i h_{iik}=0\).
Passing to the limit gives \(\sum_i\bar\lambda_i\bar h_{iik}=0\).
Therefore \eqref{eq:limit-firsttrace} holds.

We now derive the second-order trace relations. By the second identity in
\eqref{eq:Hderivatives},
\[
  H_{kl} = \sum_jh_{klj}\langle X,e_j\rangle +h_{kl} -H\sum_jh_{kj}h_{jl}.
\]
The first term tends to zero, since $h_{klj}$ is uniformly bounded and
$|X^\top(q_m)|\to0$. In the principal frame,
\[
  h_{kl}=\lambda_k\delta_{kl}, \qquad \sum_jh_{kj}h_{jl} = \lambda_k^2\delta_{kl}.
\]
Hence
\[
  H_{kl}(q_m) \to (\bar\lambda_k-\bar H\bar\lambda_k^2)\delta_{kl}.
\]
Since \(H_{kl} = \sum_i h_{iikl}\),
we obtain
\[
  \sum_i\bar h_{iikl} = (\bar\lambda_k-\bar H\bar\lambda_k^2)\delta_{kl},
\]
which gives \eqref{eq:limit-secondtrace}.

Finally, differentiating $S$ twice gives
\[
  0 = \frac12S_{kl} = \sum_{i,j}h_{ijk}h_{ijl} + \sum_{i,j}h_{ij}h_{ijkl}.
\]
At $q_m$, where $h_{ij}=\lambda_i\delta_{ij}$, the above equation becomes
\[
  \sum_i\lambda_i h_{iikl} = -\sum_{i,j}h_{ijk}h_{ijl}.
\]
Passing to the limit gives
\[
  \sum_i\bar\lambda_i\bar h_{iikl} = -\sum_{i,j}\bar h_{ijk}\bar h_{ijl}.
\]
Thus, defining \(\bar Q_{kl} = \sum_{i,j}\bar h_{ijk}\bar h_{ijl}\),
we obtain \eqref{eq:limit-weightedtrace}.

The full symmetry of $\bar h_{ijk}$ and the symmetry of $\bar h_{ijkl}$
in the first three indices follow by passing to the limit. Likewise, the Ricci identity in the first line of
\eqref{eq:ricci-general} and the identities \eqref{eq:bochner1}
and \eqref{17} remain valid in the limit, since they involve only
these bounded components and the principal curvatures.
\end{proof}

In the subsequent algebraic argument we omit the bars from
these limiting quantities. Thus \(h_{ijk}\) and \(h_{ijkl}\)
there denote the limiting components.

\section{An algebraic inequality}
\label{sec:algebra}\label{sec:positive-minimum}

We now turn to the algebraic inequality needed for the proof of
Theorem~\ref{thm:main}. By Proposition~\ref{prop:limit},
the limiting principal curvatures are positive and satisfy
$\bar H\bar\lambda_i\geq 1$.

We formulate the inequality for algebraic quantities satisfying
the relations arising from Proposition~\ref{prop:limit}.
Throughout this section, we omit the bars and use
$\lambda_i$, $h_{ijk}$, and $h_{ijkl}$ for the corresponding
algebraic variables. In particular, the Gauss equation and
the Ricci identity give the commutation relation
\begin{equation}\label{eq:algebra-ricci}
h_{ijkl}-h_{ijlk}
=
\lambda_i\lambda_j(\lambda_i-\lambda_j)
(\delta_{ik}\delta_{jl}-\delta_{il}\delta_{jk}).
\end{equation}

After relabeling, we may assume that
\[
0<\lambda_1\leq\lambda_2\leq\lambda_3.
\]
For convenience, put
\[
a=\lambda_1,\qquad b=\lambda_2,\qquad c=\lambda_3,
\]
so that $A=\operatorname{diag}(a,b,c)$, and set
\[
H=a+b+c,\qquad
S=a^2+b^2+c^2,\qquad
f_3=a^3+b^3+c^3.
\]
Here $H$, $S$, and $f_3$ are regarded as algebraic functions
of $a,b,c$, agreeing with the corresponding limiting curvature
quantities in the geometric case.

To make the algebraic conditions compatible with rescaling,
we replace the constant $1$ by a parameter $\rho>0$ and put
\[
C=S(S-\rho).
\]
We assume $S>\rho$, so that $C>0$. Since $a$ is the smallest
principal curvature and $H>0$, the inequalities
$H\lambda_i\geq\rho$ for all $i$ are equivalent to
$aH\geq\rho$.

Under the rescaling
\[
\lambda_i\mapsto t\lambda_i,\qquad
h_{ijk}\mapsto t^2h_{ijk},\qquad
h_{ijkl}\mapsto t^3h_{ijkl},\qquad
\rho\mapsto t^2\rho
\qquad (t>0),
\]
we have $H\mapsto tH$, $S\mapsto t^2S$, and
$C\mapsto t^4C$. Thus the algebraic constraints in
Proposition~\ref{prop:positivity} retain their form.
This scaling will later allow us to normalize $a=1$.
The geometric situation of Proposition~\ref{prop:limit}
corresponds to $\rho=1$.

Define
\begin{equation}\label{eq:J}
\begin{aligned}
J={}&
\sum_{i,j,k,l}
\left(\lambda_l+\frac H2\right)h_{ijkl}^2
-3\sum_{i,j,k,r,s}h_{ijk}h_{rsk}h_{ijrs}\\
&+\sum_{i,j,k}
\left\{
3S\lambda_i+6\lambda_i^2\lambda_j-3\lambda_i^3
-\frac{3H}{2}(\lambda_i^2-2\lambda_i\lambda_j)
\right\}h_{ijk}^2\\
&-\left\{f_3+\frac H2(S-2\rho)\right\}C.
\end{aligned}
\end{equation}

For the limiting quantities obtained in
Proposition~\ref{prop:limit}, the identity
\eqref{eq:bochner1} gives
\[
0=
\sum_{i,j,k,l}h_{ijkl}^2
-3\sum_{i,j,k}
(\lambda_i^2-2\lambda_i\lambda_j)h_{ijk}^2
-(S-2)S(S-1).
\]
These limiting quantities correspond to $\rho=1$ in
the algebraic formulation, so that $C=S(S-1)$.
Multiplying the above identity by $H/2$ and adding
\eqref{17}, we obtain $J=0$.

The corresponding algebraic positivity result is as follows.

\begin{proposition}\label{prop:positivity}
Under the notation above, suppose that $A$ is nonscalar
and $aH\geq\rho$. Let $h_{ijk}$ be fully symmetric and satisfy
\begin{equation}\label{eq:algebra-T}
\sum_i h_{iik}=0,\qquad
\sum_i\lambda_i h_{iik}=0,\qquad
\sum_{i,j,k}h_{ijk}^2=C.
\end{equation}
Let $h_{ijkl}$ be symmetric in its first three indices
and satisfy \eqref{eq:algebra-ricci} and the trace identities
\begin{equation}\label{eq:algebra-U}
\begin{aligned}
\sum_i h_{iikl}
&=(\rho\lambda_k-H\lambda_k^2)\delta_{kl},\\
\sum_i\lambda_i h_{iikl}
&=-\sum_{r,s}h_{rsk}h_{rsl}.
\end{aligned}
\end{equation}
Then $J>0$.
\end{proposition}

We will prove Proposition~\ref{prop:positivity} in three steps.
In the first step, we use the Ricci commutation relation
to decompose the fourth-order tensor and the trace identities
to describe the remaining free components.
In the second step, we express $J$ in terms of these
components and reduce the problem to the minimization
of a quadratic form.
In the final step, we establish the strict positivity
of the resulting minimum, considering separately the cases
of distinct and repeated principal curvatures.

\medskip
\textbf{Step 1. Tensor decomposition and trace equations.}
\quad

Let $\widehat h_{ijk}$ be any
fully symmetric array satisfying the first two equations in
\eqref{eq:algebra-T}. Set
\[
  M=\sum_{i,j,k}\widehat h_{ijk}^{\,2}>0.
\]
Since $M>0$, we may normalize $\widehat h$ by setting
\[
  h_{ijk} = \sqrt{\frac{C}{M}}\, \widehat h_{ijk}.
\]
Then \(\sum_{i,j,k}h_{ijk}^{\,2}=C\),
while the first two equations in \eqref{eq:algebra-T} are preserved
under this rescaling. Hence $h_{ijk}$ satisfies all three conditions in
\eqref{eq:algebra-T}. Conversely, every array satisfying
\eqref{eq:algebra-T} is obtained in this way from a nonzero fully
symmetric array $\widehat h_{ijk}$ satisfying the first two trace
constraints.

Set
\[
  \widehat Q_{kl}=\sum_{r,s}\widehat h_{rsk}\widehat h_{rsl}
\]
and $q_i=\widehat Q_{ii}$. We have
\begin{equation}\label{eq:raw-gram}
 \sum_iq_i=M,\qquad
 \sum_{r,s}h_{rsk}h_{rsl}=\frac CM\widehat Q_{kl}.
\end{equation}

The array $h_{ijkl}$ is symmetric in the first three indices, but not necessarily symmetric in its last two indices. Instead,
the difference $h_{ijkl}-h_{ijlk}$ is prescribed by \eqref{eq:algebra-ricci}.
For each $i<j$, define
\begin{equation*}
\begin{aligned}
 \Omega_{iijj}=\Omega_{ijij}=\Omega_{jiij}
  &=\frac12\lambda_i\lambda_j(\lambda_i-\lambda_j),\\
 \Omega_{jjii}=\Omega_{jiji}=\Omega_{ijji}
  &=-\frac12\lambda_i\lambda_j(\lambda_i-\lambda_j),
\end{aligned}
\end{equation*}
and set all other components equal to zero. This defines every
component of $\Omega$ and gives symmetry in the first three indices.
If $i\ne j$, the difference $\Omega_{ijkl}-\Omega_{ijlk}$ can be
nonzero only when $(k,l)=(i,j)$ or $(j,i)$. For example,
\[
  \Omega_{ijij}-\Omega_{ijji}=\lambda_i\lambda_j(\lambda_i-\lambda_j).
\]
The other case has the opposite sign. When $i=j$, the difference is zero.
Hence
\begin{equation}\label{eq:omega-commutation}
 \Omega_{ijkl}-\Omega_{ijlk}
 =\lambda_i\lambda_j(\lambda_i-\lambda_j)
   (\delta_{ik}\delta_{jl}-\delta_{il}\delta_{jk}).
\end{equation}
Put
\begin{equation}\label{eq:decomposition}
 V_{ijkl}=M(h_{ijkl}-\Omega_{ijkl}).
\end{equation}
Subtracting \eqref{eq:omega-commutation} from \eqref{eq:algebra-ricci},
we obtain $V_{ijkl}=V_{ijlk}$. Since $V$ is also symmetric in the
first three indices, $V$ is invariant under $(12)$, $(23)$, and $(34)$.
These adjacent transpositions generate all permutations of four indices,
which proves that $V$ is fully symmetric. Conversely,
\[
  h_{ijkl}=\Omega_{ijkl}+M^{-1}V_{ijkl}
\]
satisfies \eqref{eq:algebra-ricci}
for every fully symmetric $V$.

We next calculate the two traces. If $k\ne l$, then
$\Omega_{iikl}=0$. For $k=l$, we have
\[
\begin{aligned}
 \sum_i\Omega_{iikk}
 &=\frac12\left(\lambda_k\sum_i\lambda_i^2-\lambda_k^2\sum_i\lambda_i\right)
   =\frac12(S\lambda_k-H\lambda_k^2),\\
 \sum_i\lambda_i\Omega_{iikk}
 &=\frac12\left(\lambda_k\sum_i\lambda_i^3-\lambda_k^2\sum_i\lambda_i^2\right)
   =\frac12(f_3\lambda_k-S\lambda_k^2).
\end{aligned}
\]
The summands vanish when $i=k$, so both sums include all $i$.
Combining \eqref{eq:algebra-U} and \eqref{eq:decomposition}, we obtain
\[
\begin{aligned}
 \sum_iV_{iikl}
 &=M\left\{\rho\lambda_k-H\lambda_k^2
            -\frac12(S\lambda_k-H\lambda_k^2)\right\}\delta_{kl},\\
 \sum_i\lambda_iV_{iikl}
 &=M\left\{-\frac CM\widehat Q_{kl}
            -\frac12(f_3\lambda_k-S\lambda_k^2)\delta_{kl}\right\},
\end{aligned}
\]
where the second equality uses \eqref{eq:raw-gram}. Set
\begin{equation*}
 R_i=\left(\rho-\frac S2\right)\lambda_i-\frac H2\lambda_i^2.
\end{equation*}
The trace equations therefore become
\begin{equation}\label{eq:scaled-traces}
\begin{aligned}
 \sum_iV_{iikl}&=MR_k\delta_{kl},\\
 \sum_i\lambda_iV_{iikl}
 &=-C\widehat Q_{kl}+\frac M2(S\lambda_k^2-f_3\lambda_k)\delta_{kl}.
\end{aligned}
\end{equation}
For $k=l$, subtracting $\lambda_k$ times the first equality from
the second eliminates $V_{kkkk}$. The right-hand side is
\[
\begin{aligned}
 -Cq_k+\frac M2(S\lambda_k^2-f_3\lambda_k)-M\lambda_kR_k
 &=M\left\{(S-\rho)\lambda_k^2+
     \frac{H\lambda_k^3-f_3\lambda_k}{2}\right\}-Cq_k.
\end{aligned}
\]
We write
\begin{equation}\label{eq:xi}
 \xi_i=M\left\{(S-\rho)\lambda_i^2+
       \frac{H\lambda_i^3-f_3\lambda_i}{2}\right\}-Cq_i.
\end{equation}
Summing over $i$ and using $\sum_iq_i=M$, we have
\[
  \sum_i\xi_i =M\left\{(S-\rho)S+\frac12(Hf_3-f_3H)\right\}-CM=0.
\]
Thus
\begin{equation}\label{eq:compatibility}
 \sum_i\xi_i=MC-C\sum_iq_i=0,
\end{equation}
which is the compatibility condition for \eqref{eq:scaled-traces}.

\begin{lemma}\label{lem:rank}
For a nonscalar \(A\), the twelve equations
\eqref{eq:scaled-traces} have rank eleven, and
\eqref{eq:compatibility} is the only compatibility relation.
The solution space is a product of four affine lines.
\end{lemma}

\begin{proof}
The components with all indices occurring an even number of
times are
\[
 (V_{1111},V_{2222},V_{3333},V_{1122},V_{1133},V_{2233}).
\]
The ordinary trace equations give
\[
  V_{iiii}=MR_i-\sum_{j\ne i}V_{iijj},\quad i=1,2,3.
\]
Substituting these values into the weighted traces, we obtain
\[
\begin{aligned}
 (b-a)V_{1122}+(c-a)V_{1133}&=\xi_1,\\
 (a-b)V_{1122}+(c-b)V_{2233}&=\xi_2,\\
 (a-c)V_{1133}+(b-c)V_{2233}&=\xi_3.
\end{aligned}
\]
Thus the coefficient matrix is
\[
 \begin{pmatrix}
 b-a&c-a&0\\
 a-b&0&c-b\\
 0&a-c&b-c
 \end{pmatrix}
\]
on the last three components. The sum of the three rows is zero.
If \(c>b\), the minor \((c-a)(c-b)\) is positive.
If \(c=b>a\), the minor \((a-b)(a-c)\) is positive.
Thus these six components are subject to five independent
equations, and the system is consistent if and only if
\eqref{eq:compatibility} holds.

For each pair \(i<j\), the remaining group has components
\((V_{11ij},V_{22ij},V_{33ij})\). The two coefficient rows
are \((1,1,1)\) and \((a,b,c)\), which are independent.
The total rank is therefore \(5+3\cdot2=11\).
\end{proof}

\Needspace{7\baselineskip}
\medskip
We now parametrize the third-order components. Put
\begin{equation*}
 d=(b-c,c-a,a-b),\quad
 D_0=\sum_i d_i^2,\quad J_2=\sum_i\lambda_i d_i^2.
\end{equation*}
All solutions of the first two equations in
\eqref{eq:algebra-T} have the form
\begin{equation}\label{eq:T-param}
 \widehat h_{iik}=d_i u_k,\qquad
 \widehat h_{123}=u_4,
\end{equation}
with full symmetry. Indeed, for each \(k\), the vector \((\widehat h_{11k},\widehat h_{22k},\widehat h_{33k})\)
is orthogonal to both \((1,1,1)\) and \((a,b,c)\), whose
common orthogonal complement is spanned by \(d\). The
component \(\widehat h_{123}\) is free. A component with two equal indices occurs three times in the
norm, and the component $\widehat h_{123}$ occurs six times. We obtain
\begin{equation}\label{eq:M}
 M=\sum_{k=1}^3(3D_0-2d_k^2)u_k^2+6u_4^2.
\end{equation}
Since $3D_0-2d_k^2=d_k^2+3\sum_{i\ne k}d_i^2>0$, the quadratic
form in \eqref{eq:M} is positive definite.

Define the fully symmetric array
\begin{equation*}
 \widehat\Phi_{ijkl}=\frac13\sum_m
 \bigl(\widehat h_{ijm}\widehat h_{klm}
       +\widehat h_{ikm}\widehat h_{jlm}
       +\widehat h_{ilm}\widehat h_{jkm}\bigr).
\end{equation*}

\begin{lemma}\label{lem:gram}
Suppose that \(a<b<c\), and put \(\theta=u_1^2+u_2^2+u_3^2\).
For pairwise distinct \(i,j,k\), one has
\begin{align}
 q_i&=D_0u_i^2+2d_i^2\sum_{j\ne i}u_j^2+2u_4^2,
 \qquad
 \widehat Q_{ij}=2d_k(d_ku_iu_j-u_ku_4),
 \label{eq:gram-components}\\
 \widehat\Phi_{iiii}&=d_i^2\theta,\qquad
 \widehat\Phi_{iijj}=\frac{q_i+q_j-q_k}{3}+d_id_j\theta,
 \label{eq:Phi-even}\\
 \widehat\Phi_{iiij}&=-\frac{d_i}{2d_k}\widehat Q_{ij},
 \qquad
 \widehat\Phi_{ijjj}=-\frac{d_j}{2d_k}\widehat Q_{ij},
 \qquad
 \widehat\Phi_{ijkk}=\frac16\widehat Q_{ij}.
 \label{eq:Phi-odd}
\end{align}
Moreover,
\begin{equation}\label{eq:theta-Q}
 4d_1d_2d_3\theta=\sum_i d_iq_i.
\end{equation}
Thus, for fixed distinct principal curvatures, \(\widehat\Phi\) depends linearly on \(\widehat Q\).
\end{lemma}

\begin{proof}
Substituting \eqref{eq:T-param} into the definition of
$\widehat Q_{kl}$, we obtain \eqref{eq:gram-components}. For instance, the coefficient of
\(u_iu_j\) in \(\widehat Q_{ij}\) is \(D_0+2d_id_j=2d_k^2\), and the coefficient of \(u_ku_4\)
is \(2(d_i+d_j)=-2d_k\). Since \(D_0-2d_i^2=-2d_jd_k\),
\[
  q_i=2d_i^2\theta-2d_jd_ku_i^2+2u_4^2.
\]
Multiplication by \(d_i\) and summation, using
\(\sum_i d_i=0\) and \(\sum_i d_i^3=3d_1d_2d_3\), gives
\eqref{eq:theta-Q}.

The definition of \(\widehat\Phi\) gives
\[
  \widehat\Phi_{iiii}=\sum_m\widehat h_{iim}^2=d_i^2\theta.
\]
The trace condition on \(\widehat h\) also gives
\[
  \sum_i\widehat\Phi_{iikl}=\frac23\widehat Q_{kl}.
\]
Apply this to the three diagonal entries and solve for \(\widehat\Phi_{1122},\widehat\Phi_{1133},\widehat\Phi_{2233}\)
to obtain \eqref{eq:Phi-even}. Finally,
\[
  \widehat\Phi_{iiij} =d_i\{(d_i+d_j)u_iu_j+u_ku_4\} =-\frac{d_i}{2d_k}\widehat Q_{ij}.
\]
The formula for \(\widehat\Phi_{ijjj}\) follows in the same
way. In \(3\widehat\Phi_{ijkk}\), the first contraction is
\(-\widehat Q_{ij}/2\), and the other two are each
\(\widehat Q_{ij}/2\). This proves \eqref{eq:Phi-odd}.
\end{proof}

\Needspace{8\baselineskip}
The following identities will simplify the terms involving
\(\widehat h_{ijk}\). They also hold when two principal
curvatures coincide.

\begin{lemma}\label{lem:contractions-Q}
For a fully symmetric \(\widehat h_{ijk}\) satisfying the two
trace conditions in \eqref{eq:algebra-T}, one has
\begin{equation*}
\begin{aligned}
 \sum_{ijk}\lambda_i\lambda_j\widehat h_{ijk}^2
 &=\sum_i\left\{2H\lambda_i-2\lambda_i^2
                    -\frac{H^2-S}{2}\right\}q_i,\\
 \sum_{ijk}\lambda_i^2\lambda_j\widehat h_{ijk}^2
 &=\sum_i\left\{-H\lambda_i^2+H^2\lambda_i
                    -\frac{2H^3-3HS+f_3}{6}\right\}q_i.
\end{aligned}
\end{equation*}
\end{lemma}

\begin{proof}
For fixed $k$, let $B=(\widehat h_{ijk})_{i,j}$.
Then $\operatorname{tr}B=\operatorname{tr}(AB)=0$.
Recall that the Cayley--Hamilton identity for a
$3\times3$ matrix $T$ is
\[
T^3-(\operatorname{tr}T)T^2
+\frac12\bigl((\operatorname{tr}T)^2
-\operatorname{tr}(T^2)\bigr)T
-(\det T)I=0.
\]
Applying this identity to $T=A+sB$ and comparing
the coefficients of $s^2$, we obtain
\[
\begin{aligned}
0={}&AB^2+BAB+B^2A-HB^2
-\frac12\operatorname{tr}(B^2)A-\left\{\operatorname{tr}(AB^2)
-\frac H2\operatorname{tr}(B^2)\right\}I.
\end{aligned}
\]
Multiply by \(A\) and take the trace. Summation over \(k\)
gives
\[
 \sum_{ijk}(2\lambda_i^2+\lambda_i\lambda_j)
                       \widehat h_{ijk}^2
 =\sum_i\left\{2H\lambda_i-\frac{H^2-S}{2}\right\}q_i,
\]
which proves the first identity. Multiplication by \(A^2\)
instead gives
\[
 \sum_{ijk}(2\lambda_i^3+\lambda_i^2\lambda_j)
                       \widehat h_{ijk}^2
 =\sum_i\left\{H\lambda_i^2+S\lambda_i
                     +\frac{f_3-HS}{2}\right\}q_i.
\]
Now use
\[
  \lambda_i^3-H\lambda_i^2+\frac{H^2-S}{2}\lambda_i-abc=0, \qquad 6abc=H^3-3HS+2f_3.
\]
This yields the second identity.
\end{proof}

\medskip
\noindent\textbf{Step 2. Expansion and minimization.}
\quad

Put \(w_i=\lambda_i+H/2>0\).
The correction \(\Omega\) has zero contraction with \(\widehat h_{ijk}\widehat h_{rsk}\). Indeed, for \(i<j\),
the positions \((iijj,ijij,jiij)\) and
\((jjii,jiji,ijji)\) have opposite corrections. The
contractions over either set of three positions give
\[
  \sum_m\bigl(\widehat h_{iim}\widehat h_{jjm} +2\widehat h_{ijm}^2\bigr).
\]
 Since \(V\) is fully
symmetric, it follows that
\begin{equation*}
 M^2\sum_{ijkrs}h_{ijk}h_{rsk}h_{ijrs}
       =C\sum_{ijrs}V_{ijrs}\widehat\Phi_{ijrs}.
\end{equation*}
Set
\begin{equation*}
\begin{aligned}
 K_\Omega&=\frac34\sum_{i<j}(w_i+w_j)
                  \lambda_i^2\lambda_j^2(\lambda_i-\lambda_j)^2,\\
 \kappa_i&=-\frac{33H}{2}\lambda_i^2
      +\frac{27H^2+3S}{2}\lambda_i
      -\frac{5H^3}{2}+HS+H\rho-9abc.
\end{aligned}
\end{equation*}
Applying Lemma~\ref{lem:contractions-Q} to the terms independent
of \(V\), we obtain
\begin{align*}
 &\sum_{ijk}
 \bigl\{S-2\rho+3(\lambda_i^2-2\lambda_i\lambda_j)\bigr\}
       \widehat h_{ijk}^2
 =\sum_i(15\lambda_i^2-12H\lambda_i+3H^2-2S-2\rho)q_i,\\
 &\sum_{ijk}
 (3S\lambda_i+6\lambda_i^2\lambda_j-3\lambda_i^3-f_3)
       \widehat h_{ijk}^2
 =\sum_i\left\{-9H\lambda_i^2
      +\frac32(5H^2+S)\lambda_i-H^3-9abc\right\}q_i.
\end{align*}
Subtracting \(H/2\) times the first identity from the second
gives \(\sum_i\kappa_iq_i\). The fixed correction \(\Omega\)
contributes \(K_\Omega M^2\). Consequently,
\begin{equation}\label{eq:quadratic}
 M^2J
 =K_\Omega M^2+CM\sum_i\kappa_iq_i
       +\sum_\alpha(m_\alpha V_\alpha^2+\ell_\alpha V_\alpha),
\end{equation}
where the last sum runs over the fifteen independent
components of $V$. The corresponding quadratic and
linear coefficients are given by
\[
\begin{aligned}
m_{iiii}&=w_i,\\
m_{iijj}&=3(w_i+w_j), && i<j,\\
m_{iiij}&=3w_i+w_j, && i\ne j,\\
m_{ijkk}&=3(w_i+w_j+2w_k),
&& i,j,k\text{ distinct},\ i<j,
\end{aligned}
\]
and
\[
\begin{aligned}
\ell_{iiii}&=-3C\widehat\Phi_{iiii},\\
\ell_{iijj}&=-18C\widehat\Phi_{iijj}
-3M\lambda_i\lambda_j(\lambda_i-\lambda_j)^2,
&& i<j,\\
\ell_{iiij}&=-12C\widehat\Phi_{iiij},
&& i\ne j,\\
\ell_{ijkk}&=-36C\widehat\Phi_{ijkk},
&& i,j,k\text{ distinct},\ i<j.
\end{aligned}
\]
For example, $V_{iijj}$ has six distinct permutations,
three with last index $i$ and three with last index $j$.
Hence its quadratic coefficient is $3(w_i+w_j)$.
The cross term involving $\Omega$ contributes
\[
3M(w_j-w_i)\lambda_i\lambda_j(\lambda_i-\lambda_j)V_{iijj}
=-3M\lambda_i\lambda_j(\lambda_i-\lambda_j)^2V_{iijj}.
\]
The contraction with $\widehat\Phi$ contributes
$-18C\widehat\Phi_{iijj}V_{iijj}$.
These give the stated coefficients $m_{iijj}$ and
$\ell_{iijj}$.
The remaining coefficients follow similarly from the
multiplicities $1$, $4$, and $12$ of
$V_{iiii}$, $V_{iiij}$, and $V_{ijkk}$, respectively.

Since $w_i>0$, all $m_\alpha$ are positive.
By Lemma~\ref{lem:rank}, the minimization under
\eqref{eq:scaled-traces} reduces to completing four
scalar squares.

Fix a nonzero array $(\widehat h_{ijk})$, and let
$\mathcal Q=\min M^2J$, where the minimum is taken
over all $V$ satisfying \eqref{eq:scaled-traces}.
We first consider the case $a<b<c$.
The cases of two equal principal curvatures
will be treated in Lemma~\ref{lem:repeated}.

\Needspace{7\baselineskip}
\medskip
We first minimize the terms involving the off-diagonal components.

For a pair \(i<j\), let \(k\) be the remaining index and put
\[
 \delta=\lambda_j-\lambda_i>0,\qquad
 D_k=5(2H+\lambda_k)\delta^2
                  +(2H-\lambda_k)(3\lambda_k-H)^2>0.
\]
Take \(V_{kkij}\) as the free component. The two trace
equations give
\[
 V_{iiij}
 =\frac{C\widehat Q_{ij}+(\lambda_k-\lambda_j)V_{kkij}}{\delta},
 \qquad
 V_{jjij}
 =-\frac{C\widehat Q_{ij}+(\lambda_k-\lambda_i)V_{kkij}}{\delta}.
\]
Substituting these expressions into the quadratic expansion
of $M^2J$ above and using \eqref{eq:Phi-odd}, we obtain
the following quadratic polynomial in $V_{kkij}$:
\[
\begin{split}
 &\frac{D_k}{\delta^2}V_{kkij}^2
 +\frac{C\{5(H+\lambda_k)(3\lambda_k-H)-\delta^2\}}
             {\delta^2}\widehat Q_{ij}V_{kkij}
       +\frac{2C^2(H+7\lambda_k)}{\delta^2}\widehat Q_{ij}^2.
\end{split}
\]
Since $D_k>0$, completing the square gives
\[
 \frac{D_k}{\delta^2}\left(V_{kkij}
 +\frac{C\{5(H+\lambda_k)(3\lambda_k-H)-\delta^2\}}{2D_k}
       \widehat Q_{ij}\right)^2+C^2\gamma_{ij}\widehat Q_{ij}^2.
\]
The minimum is therefore $C^2\gamma_{ij}\widehat Q_{ij}^2$, where
\begin{equation}\label{eq:gamma-odd}
 \gamma_{ij}
 =\frac{2(H+7\lambda_k)}{\delta^2}
 -\frac{\{5(H+\lambda_k)(3\lambda_k-H)-\delta^2\}^2}
             {4\delta^2D_k}.
\end{equation}

\Needspace{7\baselineskip}
\medskip
For the even components, put $x=b-a>0$ and $y=c-b>0$.\label{subsec:even}
Then \(d=(-y,x+y,-x)\).
For \(\mathbf q=(q_1,q_2,q_3)\), use
\begin{equation*}
 M=q_1+q_2+q_3,\qquad
 \theta=\frac{\sum_i d_iq_i}{4d_1d_2d_3},
\end{equation*}
and define \(\xi_i\) by \eqref{eq:xi}.
The following expressions are linear in \(q_1,q_2,q_3\):
\begin{equation}\label{eq:direct-v}
\begin{alignedat}{2}
  v_{11}&=MR_1-\frac{\xi_1}{x+y},&\qquad v_{22}&=MR_2-\frac{\xi_2}{y},\\
  v_{33}&=MR_3-\frac{\xi_1}{x+y}-\frac{\xi_2}{y},& v_{12}&=0,\\
  v_{13}&=\frac{\xi_1}{x+y},& v_{23}&=\frac{\xi_2}{y},
\end{alignedat}
\end{equation}
where \(v_{ij}\) denotes the \(iijj\) component of a
particular solution. The ordinary traces of this solution are \(MR_i\).
The differences between the weighted and ordinary trace
equations are given by
\[
 xv_{12}+(x+y)v_{13}=\xi_1,\qquad
 -xv_{12}+yv_{23}=\xi_2,\qquad
 -(x+y)v_{13}-yv_{23}=\xi_3.
\]
The first two follow from \eqref{eq:direct-v}, and the third
follows from \(\sum_i\xi_i=0\).

For convenience, write
$\ell_{ii}=\ell_{iiii}$ and
$\ell_{ij}=\ell_{iijj}$ for $i<j$.
Using the formulas for $\ell_{iiii}$ and $\ell_{iijj}$
given above, together with \eqref{eq:Phi-even}, we obtain
\begin{equation*}
\begin{aligned}
 \ell_{ii}&=-3Cd_i^2\theta,\\
 \ell_{ij}&=-3M\lambda_i\lambda_j(\lambda_i-\lambda_j)^2
            -6C(q_i+q_j-q_k)-18Cd_id_j\theta
            \quad(i<j),
\end{aligned}
\end{equation*}
where \(k\) is the remaining index. Put
\[
  m_{ii}=w_i,\qquad m_{ij}=3(w_i+w_j)\quad(i<j), \qquad n_{ij}=-d_id_j\quad(i\leq j).
\]
We extend $n_{ij}=-d_id_j$ symmetrically to all $i,j$. Since
$\sum_i d_i=\sum_i\lambda_i d_i=0$, we have
\[
\begin{aligned}
  \sum_i n_{ik}&=-d_k\sum_i d_i=0,\\
  \sum_i\lambda_i n_{ik}&=-d_k\sum_i\lambda_i d_i=0.
\end{aligned}
\]
Thus $n$ is a nonzero solution of the homogeneous even trace equations.
By Lemma~\ref{lem:rank}, every even solution has the form
$V_{iijj}=v_{ij}+sn_{ij}$, where $s\in\RR$. The coefficient of \(s^2\) is
\begin{equation}\label{eq:scalar-m0}
 m_0=\sum_iw_id_i^4+3\sum_{i<j}(w_i+w_j)d_i^2d_j^2=D_0(HD_0+2J_2)>0.
\end{equation}
To verify the second equality, use
\(\sum_i d_i=\sum_i\lambda_i d_i=0\) and
\[
  d_i^3-\frac{D_0}{2}d_i-d_1d_2d_3=0.
\]
Multiplying by \(d_i\) and by \(\lambda_i d_i\), respectively,
and summing over \(i\), we obtain
\[
  \sum_i d_i^4=\frac{D_0^2}{2},\qquad \sum_i\lambda_i d_i^4=\frac{D_0J_2}{2}.
\]
Thus the first line of \eqref{eq:scalar-m0} equals
\[
 3D_0\left(J_2+\frac H2D_0\right)
 -2\left(\frac{D_0J_2}{2}+\frac H4D_0^2\right)
 =D_0(HD_0+2J_2).
\]

Let $
L=\sum_{i\leq j}n_{ij}(2m_{ij}v_{ij}+\ell_{ij})$.
The terms involving the free parameter are
\[
  m_0s^2+Ls =m_0\left(s+\frac{L}{2m_0}\right)^2-\frac{L^2}{4m_0}.
\]
Consequently the even terms, together with the constant
terms in \eqref{eq:quadratic}, have minimum
\begin{equation}\label{eq:direct-even}
\begin{split}
 E_0(\mathbf q)
 =K_\Omega M^2+CM\sum_i\kappa_iq_i+\sum_{i\leq j}(m_{ij}v_{ij}^2+\ell_{ij}v_{ij})
       -\frac{L^2}{4m_0}.
\end{split}
\end{equation}
Since the four groups of components are independent under the trace
equations, the four minima can be attained simultaneously. Combining
the even minimum with the three off-diagonal minima, we obtain
\begin{equation}\label{eq:homogenization}
 \mathcal Q
 =E_0(\mathbf q)+C^2\sum_{i<j}\gamma_{ij}\widehat Q_{ij}^2.
\end{equation}
For fixed principal curvatures, \(E_0\) is quadratic in the
three diagonal entries. The expression \(\mathcal Q\) is
therefore homogeneous of degree four in \(u_1,u_2,u_3,u_4\).

To reduce the remaining expression, suppose that \(a<b<c\). By \eqref{eq:gram-components}, the quantities
\begin{equation}\label{eq:cone-coordinates}
\begin{aligned}
 p&=q_1-2y^2\theta=2x(x+y)u_1^2+2u_4^2,\\
 q&=q_3-2x^2\theta=2y(x+y)u_3^2+2u_4^2,\\
 r&=2(x+y)^2\theta-q_2=2xyu_2^2-2u_4^2
\end{aligned}
\end{equation}
satisfy $p,q\geq0$ and $r\geq-\min\{p,q\}$.

Indeed, \(p+r=2x((x+y)u_1^2+yu_2^2)\geq0\) and \(q+r=2y((x+y)u_3^2+xu_2^2)\geq0\).
If \(p=q=r=0\), then \eqref{eq:cone-coordinates} gives
\(u_1=u_2=u_3=u_4=0\). Thus \(M>0\) implies
\((p,q,r)\ne(0,0,0)\).

The off-diagonal entries satisfy
\begin{align}
 \widehat Q_{12}^2
 &=-\frac{x^3}{(x+y)^3}\widehat Q_{13}^2
   +\frac{x^2}{y(x+y)}p(q+r),\notag\\
 \widehat Q_{23}^2
 &=-\frac{y^3}{(x+y)^3}\widehat Q_{13}^2
   +\frac{y^2}{x(x+y)}q(p+r).
 \label{eq:eliminate-Q}
\end{align}
For the first identity, use
\[
  \widehat Q_{12}=2x(xu_1u_2+u_3u_4),\qquad \widehat Q_{13}=2(x+y)((x+y)u_1u_3-u_2u_4)
\]
and the elementary factorization
\[
\begin{split}
 &(x+y)(xu_1u_2+u_3u_4)^2
       +x((x+y)u_1u_3-u_2u_4)^2\\
 &\hspace{20mm}=(x(x+y)u_1^2+u_4^2)(xu_2^2+(x+y)u_3^2).
\end{split}
\]
The second identity follows in the same way, with
\(x,u_1\) and \(y,u_3\) interchanged.
The diagonal entries are recovered from
\begin{equation}\label{eq:recover-diagonal}
\begin{aligned}
 \theta&=\frac{yp+(x+y)r+xq}{2xy(x+y)},\\
 (q_1,q_2,q_3)
   &=(2y^2\theta+p,\ 2(x+y)^2\theta-r,\ 2x^2\theta+q).
\end{aligned}
\end{equation}
In particular, \(q_1,q_2,q_3\) depend linearly on \(p,q,r\).

Substituting \eqref{eq:eliminate-Q} into
\eqref{eq:homogenization} gives
\begin{equation}\label{eq:Q-cone-reduction}
 \mathcal Q=G(p,q,r)+C^2\beta\widehat Q_{13}^2,
 \qquad
 \beta=\gamma_{13}
       -\frac{x^3\gamma_{12}+y^3\gamma_{23}}{(x+y)^3},
\end{equation}
where
\begin{equation}\label{eq:direct-G}
\begin{split}
 G(p,q,r)={}&E_0(q_1,q_2,q_3)+C^2\left\{\frac{x^2\gamma_{12}}{y(x+y)}p(q+r)
       +\frac{y^2\gamma_{23}}{x(x+y)}q(p+r)\right\}.
\end{split}
\end{equation}
Here \(q_1,q_2,q_3\) are given by
\eqref{eq:recover-diagonal}. Since \(E_0\) is quadratic,
we may write
\begin{equation}\label{eq:G-form}
 G(p,q,r)=c_{200}p^2+c_{110}pq+c_{101}pr
         +c_{020}q^2+c_{011}qr+c_{002}r^2.
\end{equation}
The coefficients of $G$ are given by
\begin{equation}\label{eq:polarization}
\begin{aligned}
c_{200}&=G(1,0,0),&
\qquad c_{110}&=G(1,1,0)-c_{200}-c_{020},\\[5pt]
c_{020}&=G(0,1,0),&
\qquad c_{101}&=G(1,0,1)-c_{200}-c_{002},\\[5pt]
c_{002}&=G(0,0,1),&
\qquad c_{011}&=G(0,1,1)-c_{020}-c_{002}.
\end{aligned}
\end{equation}
\medskip
\noindent\textbf{Step 3. Positivity of the minimum.}
\quad

We now establish the strict positivity of the minimum
obtained in Step 2. By \eqref{eq:Q-cone-reduction}, it suffices to prove
that $G>0$ and $\beta>0$. We consider first the case
of three distinct principal curvatures, $a<b<c$.
To simplify the positivity estimates, we introduce
the following scale-invariant quantities:
\begin{equation}\label{eq:spectral-variables}
t=1-\frac{3(ab+ac+bc)}{H^2},
\qquad
w=\frac{(3a-H)(3b-H)(3c-H)}{H^3}.
\end{equation}
Since $a,b,c>0$ are distinct, we have
\[
0<t=
\frac{(a-b)^2+(b-c)^2+(c-a)^2}{2H^2}<1.
\]
These quantities will be used in the following two
lemmas to establish the required positivity.
The cases of two equal principal curvatures will
be considered separately.

\begin{lemma}\label{lem:beta}
For \(0<a<b<c\), the coefficient \(\beta\) in
\eqref{eq:Q-cone-reduction} is positive.
\end{lemma}

\begin{proof}
Using the expressions for $\gamma_{ij}$ in \eqref{eq:gamma-odd}
and the definitions of $t$ and $w$ in \eqref{eq:spectral-variables},
we obtain

\begin{equation*}
\begin{split}
 \beta={}&\frac{16xyH^8}{81(x+y)^2D_1D_2D_3}
 \bigl\{66000t^2+12460t^3+4t^4+(25+42t)w^2\\
 &\hspace{36mm}+(2000t+1670t^2)(2t-w)\bigr\}.
\end{split}
\end{equation*}
Since $t>0$ and
\[
2t-w=\frac{3}{H^3}
\left(a(b-c)^2+b(c-a)^2+c(a-b)^2\right)>0,
\]
every term in braces is nonnegative, and the term
$66000t^2$ is strictly positive.
Together with $x,y>0$ and $D_i>0$ for $i=1,2,3$,
this proves $\beta>0$.
\end{proof}

\Needspace{7\baselineskip}
\medskip
We next establish a boundary estimate for the quadratic form.

\begin{lemma}\label{lem:k3-manual}
Suppose that \(0<a<b<c\) and \(0<\rho\leq H^2/3\). Then \(G(1,1,-1)>0\).
\end{lemma}

\begin{proof}
Equation~\eqref{eq:direct-G} gives \(G(1,1,-1)=E_0(1,1,1)\).
Thus \(M=3\) and \(\theta=0\) in \eqref{eq:direct-even}.
Use \(t,w\) from \eqref{eq:spectral-variables}, and put \(u=\frac13-\frac{\rho}{H^2}\geq0\).
Besides \(0<t<1\), we have
\begin{equation}\label{eq:k3-domain}
 |w|\leq2t^{3/2},\qquad w>3t-1.
\end{equation}
Indeed,
\[
  4t^3-w^2=\frac{27(a-b)^2(b-c)^2(c-a)^2}{H^6}\geq0, \qquad w-3t+1=\frac{27abc}{H^3}>0.
\]

In these variables,
\[
  S=\frac{H^2}{3}(1+2t),\qquad C=\frac{H^4}{9}(1+2t)(2t+3u).
\]
Also \(D_0=2H^2t\) and \(J_2=H^3(2t-w)/3\), so that
\[
  m_0=\frac43H^5t(5t-w),\qquad \sum_i\kappa_i=H^3(2-5t-w-3u).
\]
Substitution in \eqref{eq:direct-even} gives
\begin{equation}\label{eq:k3-three-polynomials}
 G(1,1,-1)=\frac{H^7\{F_0(t,w)+12uF_1(t,w)+36u^2F_2(t,w)\}}
              {3888t(5t-w)},
\end{equation}
where
\begin{align*}
 F_0={}&t^2(3500+65200t+65856t^2-66368t^3-23104t^4)\\
    &+tw(4300+12520t-14768t^2-8064t^3)\\
    &-w^2(3025+10616t+2672t^2-9472t^3)
      -(440+240t)w^3-16w^4,\\
 F_1={}&t^2(6320+14512t-8480t^2-3952t^3)\\
    &+tw(1566+1278t-1112t^2)
      -(935+964t-1256t^2)w^2-68w^3,\\
 F_2={}&t^2(2354+2752t-676t^2)+tw(638-64t)
      -(289+12t)w^2.
\end{align*}
The denominator is positive by \eqref{eq:k3-domain}.
Since \(638-64t>0\), the bounds \(w\geq-2t\) and
\(w^2\leq4t^3\) give
\begin{equation*}
\begin{aligned}
 F_2&\geq t^2(2354+2752t-676t^2)
             -2t^2(638-64t)-4t^3(289+12t)\\
    &=1078t^2+1724t^3-724t^4
     \geq1078t^2+1000t^3>0.
\end{aligned}
\end{equation*}

For fixed \(t\), the other two polynomials are strictly
concave in \(w\). Indeed, on the interval in
\eqref{eq:k3-domain},
\begin{align*}
 (F_0)_{ww}
 &=-6050-21232t-5344t^2+18944t^3
      -(2640+1440t)w-192w^2\\
 &\leq-6050-15952t-2464t^2+18944t^3\\
 &=-5522-(1-t)(528+16480t+18944t^2)<0,\\
 (F_1)_{ww}
 &=-1870-1928t+2512t^2-408w\\
 &\leq-1870-1112t+2512t^2
 =-470-(1-t)(1400+2512t)<0.
\end{align*}
It therefore suffices to check the endpoints
\[
 w=2t^{3/2},\qquad
 w=\begin{cases}
 -2t^{3/2},&0<t\leq1/4,\\
 3t-1,&1/4\leq t<1.
 \end{cases}
\]
The change in the lower endpoint occurs at \(t=1/4\), since \(2s^3+3s^2-1=(2s-1)(s+1)^2\) for \(s=\sqrt t\).
For the first two, write \(t=s^2\). Direct expansion gives
\[
  F_0(s^2,\pm2s^3)=4s^4(1\mp s)P(\pm s),\qquad F_1(s^2,\pm2s^3)=4s^4Q(\pm s),
\]
where
\begin{align*}
 P(s)={}&875+3025s+16300s^2+22560s^3+28408s^4\\
        &+20144s^5+880s^6-3632s^7,\\
 Q(s)={}&1580+783s+2693s^2+639s^3-3084s^4-692s^5+268s^6.
\end{align*}
For \(0<s<1\), we have $20144s^5-3632s^7\geq16512s^5>0$, and \(Q(s)\geq1580-1083s^2\geq497\).
For \(0<s\leq1/2\), using \(s^3\leq s^2/2\) and
\(s^5\leq s^4/2\), we obtain
\[
  P(-s)\geq875-3025s+5020s^2>0,\qquad Q(-s)\geq\frac{2377}{2}+\frac{3205}{2}s^2>0.
\]
The quadratic in the first inequality is positive because \(3025^2<4\cdot875\cdot5020\).
Finally,
\begin{align*}
 F_0(t,3t-1)={}&(1-t)
 (23104t^5+28416t^4+87456t^3+43640t^2-2895t-2601),\\
 F_1(t,3t-1)={}&-3952t^5-512t^4+1410t^3+10201t^2+2468t-867.
\end{align*}
When \(1/4\leq t<1\), we have \(t^3\geq t^2/4\), \(t\leq4t^2\), and \(1\leq16t^2\).
The polynomial in parentheses is therefore at least
\[
  (21864+43640-11580-41616)t^2=12308t^2.
\]
Furthermore,
\[
  F_1(t,3t-1)\geq5737t^2+2468t-867\geq\frac{1737}{16}>0.
\]
This shows \(F_0,F_1>0\) throughout \eqref{eq:k3-domain}.
Equation~\eqref{eq:k3-three-polynomials} now gives \(G(1,1,-1)>0\).
\end{proof}

\Needspace{7\baselineskip}
\medskip
It remains to verify the coefficient inequalities.

\begin{lemma}\label{lem:coefficients}
If $0<a<b<c$ and $0<\rho\leq aH$, then
the six coefficients of $G$ satisfy
\[
\begin{gathered}
c_{200},\,c_{020},\,c_{002},
\,c_{110},\,c_{101},\,c_{011}>0,\\[6pt]
2c_{200}+c_{110}-c_{101}>0,\qquad
c_{110}+2c_{020}-c_{011}>0,\\[6pt]
c_{200}+c_{110}+c_{020}
-c_{101}-c_{011}+c_{002}>0.
\end{gathered}
\]
\end{lemma}

\begin{proof}
Since $a<H/3$ and $\rho\leq aH$, we have
$\rho<H^2/3$. Thus the last inequality follows from
Lemma~\ref{lem:k3-manual}.

To prove the remaining eight inequalities, consider the
rescaling
\begin{equation}\label{eq:scaling}
\begin{aligned}
h_{ij}&\mapsto a^{-1}h_{ij},&
h_{ijk}&\mapsto a^{-2}h_{ijk},\\
h_{ijkl}&\mapsto a^{-3}h_{ijkl},&
\rho&\mapsto a^{-2}\rho.
\end{aligned}
\end{equation}
This rescaling preserves the hypotheses of
Proposition~\ref{prop:positivity} and sends $J$ to
$a^{-7}J$. Each coefficient in \eqref{eq:G-form}
is multiplied by the same positive factor $a^{-7}$.
Hence their signs are unchanged, and we may assume
\begin{equation}\label{eq:xyz}
\begin{gathered}
a=1,\qquad b=1+x,\qquad c=1+x+y,\\
H=3+2x+y,\qquad \rho=H-z,
\end{gathered}
\end{equation}
where $x,y>0$ and $0\leq z<H$.
Since
\[
S-H=2x+y+2x^2+2xy+y^2>0,
\]
we have $S>\rho$ throughout this domain.

By \eqref{eq:direct-even}, \eqref{eq:direct-G},
and \eqref{eq:polarization}, the eight quantities
can be expressed as rational functions of $x,y,z$.
After clearing their positive denominators,
the resulting numerators are polynomials whose
nonzero coefficients are all positive.
Moreover, each numerator contains a positive term
independent of $z$, ensuring strict positivity
also when $z=0$.

The exact polynomial expansions and coefficient
verifications are described in
Appendix~\ref{app:polynomials}.
This proves the remaining eight inequalities.
\end{proof}

\Needspace{7\baselineskip}
\medskip
We now treat the case of repeated principal curvatures.\label{subsec:repeated}

\begin{lemma}\label{lem:repeated}
Under the hypotheses of Proposition~\ref{prop:positivity},
if  exactly two principal curvatures coincide, then \(\mathcal Q>0\).
\end{lemma}

\begin{proof}
By \eqref{eq:scaling}, we may normalize the smallest
principal curvature to $1$. Since $\rho\leq aH$,
we may write $z=H-\rho$, where $0\leq z<H$.

After relabeling, put
\[
(\lambda_1,\lambda_2,\lambda_3)
=(\lambda,\lambda,\mu),
\qquad
\delta=\mu-\lambda\neq0.
\]
The first two trace conditions in
\eqref{eq:algebra-T} imply
\[
\widehat h_{33k}=0,\qquad
\widehat h_{11k}+\widehat h_{22k}=0.
\]
Thus the cubic form associated with
$\widehat h_{ijk}$ decomposes into a harmonic cubic
on the repeated eigenspace and the transverse
coordinate multiplied by a harmonic quadratic
on that eigenspace.

The constraints and the quadratic expression
\eqref{eq:quadratic} are invariant under rotations
in the repeated eigenspace. Hence $\mathcal Q$
is rotation invariant. Under such rotations,
the two pairs of coefficients transform with
weights $3$ and $2$, respectively. Put
\[
P_1=
\frac{\widehat h_{111}^2+\widehat h_{112}^2}{\delta^2},
\qquad
P_2=
\frac{\widehat h_{113}^2+\widehat h_{123}^2}{\delta^2}.
\]
For fixed $\lambda,\mu,\rho$, completing the four squares shows
that $\mathcal Q$ is a homogeneous polynomial of degree four in
the components $\widehat h_{ijk}$. Hence no mixed angular
invariant can occur. Indeed,
such an invariant would require
$3m=2n$ for nonzero integers $m,n$ with
$|m|+|n|\leq4$, which is impossible.
Consequently, $\mathcal Q$ is a linear combination
of $P_1^2$, $P_1P_2$, and $P_2^2$.

It is therefore enough to calculate $\mathcal Q$
on the family of fully symmetric arrays determined by
\[
\begin{aligned}
\widehat h_{111}=-\delta u,\quad
\widehat h_{122}=\delta u,\quad
\widehat h_{113}=-\delta v,\quad
\widehat h_{223}=\delta v,\quad
\end{aligned}
\]
with all other independent components equal to zero.
For this family, $P_1=u^2$, $P_2=v^2$, and
\[
\begin{aligned}
M&=\delta^2(4P_1+6P_2),\\
q_1=q_2&=2\delta^2(P_1+P_2),
q_3&=2\delta^2P_2,
\widehat Q_{13}&=2\delta^2uv.
\end{aligned}
\]
The other off-diagonal entries of $\widehat Q$ vanish.

By the definition of $\widehat\Phi$, we obtain
\[
\begin{alignedat}{3}
\widehat\Phi_{1111}=\widehat\Phi_{2222}
&=\delta^2(P_1+P_2),
&\quad \widehat\Phi_{1122}
&=\frac{\delta^2}{3}(P_1-P_2),
&\quad\quad  \widehat\Phi_{1113}
&=\delta^2uv,\\[6pt]
\widehat\Phi_{1133}=\widehat\Phi_{2233}
&=\frac{2\delta^2}{3}P_2,
&\quad \widehat\Phi_{3333}=\widehat\Phi_{1333}
&=0,
&\quad \widehat\Phi_{1223}
&=\frac{\delta^2}{3}uv.
\end{alignedat}
\]
The remaining odd components vanish.
In this frame,
$
w_1=w_2=\lambda+\frac H2,
w_3=\mu+\frac H2.
$
Set
\[
\begin{aligned}
t_0&=\frac{M}{\delta}
\left\{
(S-\rho)\lambda^2
+\frac{H\lambda^3-f_3\lambda}{2}
\right\}
-2C\delta(P_1+P_2),\\
t_1&=MR_1-t_0,\quad t_2=MR_3-2t_0.
\end{aligned}
\]
The even trace equations give
\[
\begin{aligned}
V_{1111}=V_{2222}&=t_1-V_{1122},\\
V_{1133}=V_{2233}&=t_0,\quad
V_{3333}=t_2.
\end{aligned}
\]
Thus $V_{1122}$ is the only free even component.

Substituting these expressions into
\eqref{eq:quadratic} and using the coefficient
formulas given above, the terms involving
$V_{1122}$ become
\[
\begin{aligned}
&8w_1V_{1122}^2
-4(w_1t_1-3C\delta^2P_2)V_{1122}\\
&\quad=
8w_1\left(
V_{1122}
-\frac{w_1t_1-3C\delta^2P_2}{4w_1}
\right)^2
-\frac{(w_1t_1-3C\delta^2P_2)^2}{2w_1}.
\end{aligned}
\]
Since $w_1>0$, the even contribution has minimum
\[
\begin{aligned}
\mathcal Q_{\mathrm e}
={}&2w_1t_1^2+w_3t_2^2
+6(w_1+w_3)t_0^2-6C\delta^2(P_1+P_2)t_1\\
&-6\delta^2(4CP_2+M\lambda\mu)t_0-\frac{(w_1t_1-3C\delta^2P_2)^2}{2w_1}.
\end{aligned}
\]

Only the odd group corresponding to the pair
$(1,3)$ contributes. Its trace equations are
\[
\begin{aligned}
V_{1333}&=-2C\delta uv,\quad
V_{1113}+V_{1223}&=2C\delta uv.
\end{aligned}
\]
The corresponding quadratic expression is
\[
\begin{aligned}
&(7\lambda+3\mu)
\left(V_{1113}^2+3V_{1223}^2\right)
+4C^2\delta^2P_1P_2
\left\{5(\lambda+\mu)-6\delta\right\}.
\end{aligned}
\]
Using the identity
\[
\begin{aligned}
V_{1113}^2+3V_{1223}^2
={}&\frac34(V_{1113}+V_{1223})^2+\frac14(V_{1113}-3V_{1223})^2\\
={}&\frac34(2C\delta uv)^2+\frac14(V_{1113}-3V_{1223})^2,
\end{aligned}
\]
we obtain the minimum
$
5C^2\delta^2(13\lambda+\mu)P_1P_2,
$
attained when $V_{1113}=3V_{1223}$.
For the other two odd groups, the trace data
and linear terms vanish, so their minima are zero.

Combining the preceding calculations gives
\begin{equation}\label{eq:repeated-scalar-formula}
\begin{aligned}
\mathcal Q
={}&\frac32(w_1+w_3)
\lambda^2\mu^2\delta^2M^2+2CM\delta^2
\left\{
2\kappa_1(P_1+P_2)+\kappa_3P_2
\right\}\\
&+\mathcal Q_{\mathrm e}+5C^2\delta^2(13\lambda+\mu)P_1P_2.
\end{aligned}
\end{equation}

Put $h=|\mu-\lambda|>0$.
We consider the two possible cases
$
(\lambda,\mu)=(1,1+h)$
 and
$(\lambda,\mu)=(1+h,1),
$
corresponding respectively to
$
(a,b,c)=(1,1,1+h)
$ and $
(a,b,c)=(1,1+h,1+h).
$
In both cases,
\[
M=h^2(4P_1+6P_2),
\]
and expanding \eqref{eq:repeated-scalar-formula}
yields
\[
\mathcal Q
=\alpha_\pm P_1^2
+\gamma_\pm P_1P_2
+\beta_\pm P_2^2.
\]

The explicit expressions for these coefficients
are given in
Appendix~\ref{app:polynomials},
Subsection~\ref{app:repeated-polynomials}.
By \eqref{eq:alpha-minus}--\eqref{eq:gamma-plus},
we have
\[
\begin{alignedat}{3}
\alpha_-&\geq 1575h^4,
&\qquad \beta_-&\geq \frac{7875h^4}{h+5},
&\qquad \gamma_-&\geq 3780h^4,\\[6pt]
\alpha_+&\geq 1728h^5,
&\qquad \beta_+&\geq \frac{7875h^4}{4h+5},
&\qquad \gamma_+&\geq 2520h^4.
\end{alignedat}
\]
All six lower bounds are strictly positive
for $h>0$. Since $P_1,P_2\geq0$ and
$P_1+P_2>0$, we conclude that
$\mathcal Q>0$ in both cases.
\end{proof}

We first consider the case $a<b<c$.
By \eqref{eq:cone-coordinates}, we have
$p,q\geq0$ and $(p,q,r)\neq(0,0,0)$.
If $r\geq0$, Lemma~\ref{lem:coefficients} shows
that all coefficients in \eqref{eq:G-form} are
strictly positive. Hence $G>0$.

If $r<0$, we rewrite \eqref{eq:G-form} as
\[
\begin{aligned}
G={}&c_{200}(p+r)^2
+c_{110}(p+r)(q+r)
+c_{020}(q+r)^2\\
&-(2c_{200}+c_{110}-c_{101})r(p+r)-(c_{110}+2c_{020}-c_{011})r(q+r)\\
&+(c_{200}+c_{110}+c_{020}
-c_{101}-c_{011}+c_{002})r^2.
\end{aligned}
\]
Since $p+r,q+r\geq0$ and $r<0$,
Lemma~\ref{lem:coefficients} implies that
every term on the right-hand side is nonnegative,
and the last term is strictly positive.
Thus $G>0$ also when $r<0$.

Combining this with Lemma~\ref{lem:beta}
and \eqref{eq:Q-cone-reduction}, we obtain
$\mathcal Q>0$. Since
$\mathcal Q=M^2\min J$ and $M>0$,
it follows that $\min J>0$, and hence $J>0$.

If exactly two principal curvatures coincide,
Lemma~\ref{lem:repeated} gives $\mathcal Q>0$,
which again implies $J>0$.
This completes the proof of
Proposition~\ref{prop:positivity}.

\Needspace{7\baselineskip}
\section{Proof of the main theorem}\label{sec:classification}

We now combine the results of the preceding sections
to prove Theorem~\ref{thm:main}. In particular,
Propositions~\ref{prop:limit} and
\ref{prop:positivity}, together with the identities
established in Section~\ref{sec:identities}, will be used to exclude
the possibility $S>1$.

Suppose that $S>1$. By Proposition~\ref{prop:limit}, the limiting
principal curvatures satisfy
\[
0<a\leq b\leq c,\qquad aH\geq1.
\]
If they are not all equal,
Proposition~\ref{prop:positivity} applies with
$\rho=1$ and yields $J>0$.
On the other hand, as established in
Section~\ref{sec:algebra}, these limiting quantities
satisfy $J=0$, a contradiction.

If all three principal curvatures are equal, then
$\lambda_i^2=S/3$ for each $i$. Equation~\eqref{eq:bochner1} then gives
\[
  \sum_{i,j,k,l}h_{ijkl}^2 =S(S-1)(S-2)-S^2(S-1) =-2S(S-1)<0.
\]
Thus \(S>1\) is impossible.

If \(S=0\), the immersion is totally geodesic, and
\(X(M)\) lies in an affine hyperplane. The self-shrinker
equation forces this hyperplane to pass through the origin.
Since the immersion is a local isometry and $M$ is complete,
$X(M)$ is the whole hyperplane.

Suppose that \(S=1\). Equation~\eqref{eq:nablah} gives
\(\nabla A=0\). Hence the principal curvatures are constant
and the corresponding eigendistributions are parallel. Two distinct
eigendistributions have zero mixed sectional curvature, so
the Gauss equation gives \(\lambda\mu=0\) for any two
distinct principal curvatures \(\lambda,\mu\).
There is therefore one nonzero principal curvature
\(\lambda\), of multiplicity \(k\), and
\[
  k\lambda^2=1,\qquad H=k\lambda,\qquad H^2=k,\qquad R=k-1.
\]

The distribution \(\ker A\) is parallel and determines a
fixed \((3-k)\)-dimensional subspace of \(\RR^4\).
Indeed, for a section \(v\) of \(\ker A\) and a tangent
vector field \(u\),
\[
  \bar\nabla_u(dX(v))=dX(\nabla_uv)+h(u,v)\xi =dX(\nabla_uv)\in dX(\ker A).
\]
Parallel transport in \(\ker A\) thus agrees with ambient
parallel transport. After an orthogonal change of coordinates,
this fixed subspace is the second factor of \(\RR^{k+1}\times\RR^{3-k}\).
Since \(H\) is constant, \eqref{eq:Hderivatives} gives
\(AX^\top=0\). Thus \(X^\top\) lies in the second factor,
whereas \(\xi\) lies in the first. The equation \(X=X^\top-H\xi\) shows that the projection of \(X\)
onto the first factor has length \(|H|=\sqrt{k}\). Hence $X(M)$ is
contained in $\SSS^k(\sqrt{k})\times\RR^{3-k}$, and $X$, regarded as a
map into this product, is a local isometry. Since $M$ is complete and
the product is connected, geodesic continuation shows that this map
is surjective. Consequently, $X(M)$ is isometric to
$\SSS^k(\sqrt{k})\times\RR^{3-k}$.

If $M$ is not orientable, we pass to its oriented
double cover $\pi:\widetilde M\to M$ and consider
the lifted immersion $\widetilde X=X\circ\pi$.
Since $\widetilde X(\widetilde M)=X(M)$, the same
description of the immersed image holds for the
original immersion.

This completes the proof of Theorem~\ref{thm:main}.
\begin{remark}
	After completing this work, we became aware of a recent arXiv preprint by S. Guan and Y. Zhao \cite{GuanZhao2026}, in which the same result was obtained. However, our paper use different approach to establish the key algebraic inequality.
\end{remark}
\Needspace{12\baselineskip}
\appendix
\section{Polynomial coefficient formulas}\label{app:polynomials}

We collect the polynomial expressions used in
Section~\ref{sec:positive-minimum}. The first subsection concerns
distinct principal curvatures. The second gives the coefficients
for two equal principal curvatures. A complete rational
calculation is included at the end.

\subsection{Eight positive numerator polynomials}\label{app:positive-numerators}

Use \eqref{eq:xyz}, with \(x,y>0\) and \(z\geq0\), and put
\begin{equation*}
\begin{aligned}
 g&=x+y,\qquad q_0=x^2+xy+y^2,\\
 p_0&=4x^3+6x^2y+5x^2+4xy^2+5xy+y^3+5y^2,\\
 p_{12}&=7x^3+7x^2y+10x^2+4xy^2+5xy+y^3+5y^2,\\
 p_{23}&=4x^3+6x^2y+5x^2+8xy^2+5xy+3y^3+10y^2.
\end{aligned}
\end{equation*}
All these factors are positive. Write
\[
  k_1=2c_{200}+c_{110}-c_{101},\qquad k_2=c_{110}+2c_{020}-c_{011}.
\]
The positivity of \(G(1,1,-1)\) was proved in
Lemma~\ref{lem:k3-manual}.

For each function
\(f\in\{c_{200},c_{110},c_{101},c_{020},c_{011},c_{002},k_1,k_2\}\),
define \(N_f=D_f f\), where \(f\) is determined by
\eqref{eq:direct-even}, \eqref{eq:direct-G}, and
\eqref{eq:polarization}. The positive denominators are given by
\[
\begin{alignedat}{2}
 D_{c_{200}}&=16x^2g^2q_0p_0,
 &\qquad D_{c_{110}}&=8xyg^2q_0p_0p_{12}p_{23},\\[3pt]
 D_{c_{101}}&=8x^2ygq_0p_0p_{12},
 &\qquad D_{c_{020}}&=16y^2g^2q_0p_0,\\[3pt]
 D_{c_{011}}&=8xy^2gq_0p_0p_{23},
 &\qquad D_{c_{002}}&=16x^2y^2q_0p_0,\\[3pt]
 D_{k_1}&=8xgq_0p_0p_{23},
 &\qquad D_{k_2}&=8ygq_0p_0p_{12}.
\end{alignedat}
\]
Exact division shows that \(N_f\in\mathbb Q[x,y,z]\).
Every nonzero coefficient of each \(N_f\) is positive.
In the order given above, the numerators have respectively
\(273,498,381,269,377,275,321,321\) nonzero terms,
for a total of \(2715\).
Retaining one positive term independent of \(z\) in each numerator gives
\[
\begin{alignedat}{2}
 N_{c_{200}}&\geq31500x^8,
 &\qquad N_{c_{110}}&\geq1260000x^{12},\\[3pt]
 N_{c_{101}}&\geq252000x^{10},
 &\qquad N_{c_{020}}&\geq34560x^9,\\[3pt]
 N_{c_{011}}&\geq172800x^{11},
 &\qquad N_{c_{002}}&\geq34560x^9,\\[3pt]
 N_{k_1}&\geq157500x^8,
 &\qquad N_{k_2}&\geq252000x^8.
\end{alignedat}
\]
Hence each numerator is strictly positive for \(x,y>0\), including \(z=0\).

The numerator signs hold for all \(z\geq0\), a larger range
than the geometric range \(0\leq z<H\). The factors \(x\) and
\(y\) in the denominators arise from the coordinates in
\eqref{eq:recover-diagonal}. These fractions are used only
for distinct eigenvalues. Lemma~\ref{lem:repeated} treats
the two repeated-eigenvalue cases directly.

The ancillary file \path{check_coefficients.py} evaluates
\eqref{eq:direct-even}, \eqref{eq:direct-G}, and
\eqref{eq:polarization} in \(\mathbb Q(x,y,z)\). It multiplies
the eight resulting fractions by the denominators given above,
performs exact polynomial
division, and checks every coefficient. The complete
expansions are supplied in \path{polynomial_coefficients.json}.
The \(2715\) nonzero coefficients are all positive.
The accompanying instructions describe how to regenerate
and compare them. The tensor identities and the repeated-eigenvalue formulas
are also checked independently in the accompanying files.
These eight general numerator signs are the part of the proof
that uses computer assistance.

\subsection{Two equal principal curvatures}\label{app:repeated-polynomials}

The following expressions are obtained from
\eqref{eq:repeated-scalar-formula}. The parameters \(u_i\)
refer to \eqref{eq:T-param}. The two quadratic expressions
\(P_1,P_2\) are the invariant squared norms used in
Lemma~\ref{lem:repeated}.

To display the resulting polynomials, write
$
  [c_0,c_1,\ldots,c_m]_h=\sum_{j=0}^m c_jh^j.
$
Since every list below has positive entries, the corresponding polynomial is positive
for \(h>0\).

Suppose first that $(a,b,c)=(1,1,1+h)$, $h>0$.
Put
$ P_1=u_1^2+u_2^2$, $ P_2=u_3^2+u_4^2/h^2$.
Then
\[
  M=h^2(4P_1+6P_2),\quad \mathcal Q=\alpha_-P_1^2+\beta_-P_2^2+\gamma_-P_1P_2.
\]
Here $H=3+h$, $S=3+2h+h^2$,
$C=S\{h(h+1)+z\}$,
\begin{equation}\label{eq:alpha-minus}
\begin{aligned}
\alpha_-&=h^{2}\bigl\{z^2[175,599,862,606,223,23]_h\\
 &\qquad +2zh[525,2380,3557,2944,1355,324,19]_h\\
 &\qquad +h^2[1575,7161,12751,12353,7225,2439,417,15]_h\bigr\}.
\end{aligned}
\end{equation}
\begin{equation}\label{eq:beta-minus}
\begin{aligned}
\beta_-&=\frac{h^{4}}{h + 5}\bigl\{2z^2[1177,1782,1086,302,27]_h\\
 &\qquad +4z[4740,11440,12277,7329,2387,365,18]_h\\
 &\qquad +[7875,46410,102551,119248,82295,34186,7925,820,18]_h\bigr\}.
\end{aligned}
\end{equation}
\begin{equation}\label{eq:gamma-minus}
\begin{aligned}
\gamma_-&=h^{2}\bigl\{z^2[630,1501,2176,1622,642,89]_h\\
 &\qquad +2zh[1155,5804,9305,8000,3845,956,71]_h\\
 &\qquad +h^2[3780,18057,33746,33975,20400,7075,1234,53]_h\bigr\}.
\end{aligned}
\end{equation}
Suppose next that $(a,b,c)=(1,1+h,1+h)$, $h>0$.
Put
$
  P_1=u_2^2+u_3^2$, $P_2=u_1^2+u_4^2/h^2.
$
Then
\[
  M=h^2(4P_1+6P_2),\qquad \mathcal Q=\alpha_+P_1^2+\beta_+P_2^2+\gamma_+P_1P_2.
\]
Here $H=3+2h$, $S=3+4h+2h^2$,
$C=S\{2h(h+1)+z\}$,
\begin{equation}\label{eq:alpha-plus}
\begin{aligned}
\alpha_+&=h^{2}\bigl\{z^2[175,276,216,136,76,32]_h\\
 &\qquad +8zh^2[102,417,714,640,304,64]_h\\
 &\qquad +16h^3[108,534,1125,1302,880,333,56]_h\bigr\}.
\end{aligned}
\end{equation}
\begin{equation}\label{eq:beta-plus}
\begin{aligned}
\beta_+&=\frac{h^{4}}{4 h + 5}\bigl\{2z^2[1177,2926,2802,1232,206]_h\\
 &\qquad +4z[4740,20531,37309,36437,20062,5852,700]_h\\
 &\qquad +[7875,73470,261287,493488,556860,389592,165840,39264,3952]_h\bigr\}.
\end{aligned}
\end{equation}

\begin{equation}\label{eq:gamma-plus}
\begin{aligned}
\gamma_+&=h^{2}\bigl\{z^2[630,1649,2472,2200,1072,236]_h\\
 &\qquad +2zh[735,5081,13104,17264,12656,4988,856]_h\\
 &\qquad +4h^2[630,4857,14794,24035,23044,13256,4300,620]_h\bigr\}.
\end{aligned}
\end{equation}

Every coefficient in these six expressions is positive.
Retaining the lowest term independent of \(z\) gives the
lower bounds used in Lemma~\ref{lem:repeated}.

\subsection{A rational calculation}

Take \((a,b,c,\rho)=(1,2,3,1)\). Then
\[
  H=6,\qquad S=14,\qquad C=182,\qquad x=y=1,\qquad d=(-1,2,-1),
\]
and \(K_\Omega=594,\qquad \kappa=(-96,114,126)^t\).
To calculate \(c_{200}\), put \((p,q,r)=(1,0,0)\).
Equation~\eqref{eq:recover-diagonal} gives
\[
  \mathbf q=(3/2,2,1/2)^t,\qquad M=4,\qquad\theta=1/4.
\]
The six components in \eqref{eq:direct-even} are given by
\[
\begin{aligned}
 \left(m_{11},n_{11},v_{11},\ell_{11}\right)
  &=\left(4,-1,\frac{209}{2},-\frac{273}{2}\right),\\[3pt]
 \left(m_{22},n_{22},v_{22},\ell_{22}\right)
  &=\left(5,-4,108,-546\right),\\[3pt]
 \left(m_{33},n_{33},v_{33},\ell_{33}\right)
  &=\left(6,-1,\frac{329}{2},-\frac{273}{2}\right),\\[3pt]
 \left(m_{12},n_{12},v_{12},\ell_{12}\right)
  &=\left(27,2,0,-1662\right),\\[3pt]
 \left(m_{13},n_{13},v_{13},\ell_{13}\right)
  &=\left(30,-1,-\frac{281}{2},-963\right),\\[3pt]
 \left(m_{23},n_{23},v_{23},\ell_{23}\right)
  &=\left(33,2,-204,474\right).
\end{aligned}
\]
In particular,
\[
  m_0=\sum_{i\leq j}m_{ij}n_{ij}^2 =4+80+6+108+30+132=360.
\]
The odd contribution in \eqref{eq:direct-G} vanishes at this
point. The four scalar quantities in \eqref{eq:direct-even} are given by
\begin{align*}
 K_\Omega M^2+CM\sum_i\kappa_iq_i&=116520,&
 \sum_{i\leq j}m_{ij}v_{ij}^2&=2229898,\\
 \sum_{i\leq j}\ell_{ij}v_{ij}&=-57081,&
 L&=-24584.
\end{align*}
It follows that
\[
  c_{200}(1,2,3,1) =116520+2229898-57081-\frac{24584^2}{1440} =\frac{84133507}{45}>0.
\]
The general positivity of \(c_{200}\) follows from the
numerator described in Subsection~\ref{app:positive-numerators}.

\end{document}